\documentclass[sts]{imsart}

\RequirePackage{amsthm,amsmath,amsfonts,amssymb}
\usepackage{bbm}
\RequirePackage[authoryear]{natbib}%% uncomment this for author-year citations
\RequirePackage[colorlinks,citecolor=blue,urlcolor=blue]{hyperref}%% uncomment this for coloring bibliography citations and linked URLs
\RequirePackage{graphicx}%% uncomment this for including figures
\usepackage{enumerate}

\startlocaldefs
\theoremstyle{plain}
\newtheorem{thm}{Theorem}
\newtheorem{lem}[thm]{Lemma}

\newtheorem{cor}[thm]{Corollary}
\theoremstyle{remark}
\newtheorem{ex}{Example}
\newcommand{\mb}{\mathbf}
\newcommand{\mbb}{\boldsymbol}

\newcommand{\abs}[1]{\left| #1\right|}

\newcommand{\ind}{\mathbbm{1}}
\newcommand{\pr}{\mathbb{P}}
\newcommand{\R}{\mathbb{R}}
\newcommand{\E}{\mathbb{E}}

\newcommand{\Var}{\mathrm{Var}}
\DeclareMathOperator*{\argmin}{arg\,min}

\newcommand{\sgn}{\mathrm{sgn}}
\newcommand{\tr}{\mathrm{tr}}
\newcommand{\iid}{\stackrel{\mathrm{i.i.d.}}{\sim}}
\newcommand{\indist}{\stackrel{d}{\to}}
\newcommand{\inprob}{\stackrel{p}{\to}}

\newcommand{\ci}{\mbox{${}\perp\mkern-11mu\perp{}$}}

\newcommand{\OLS}{\text{OLS}}
\newcommand{\DB}{\text{DB}}
\newcommand{\DEF}{\text{DEF}}
\newcommand{\WDEF}{\text{W-DEF}}

\newcommand{\GENDEF}{\text{GLM-DEF}}
\newcommand{\GLM}{\text{GLM}}
\newcommand{\re}{\text{re}}
\newcommand{\sq}{\text{sq}}

\endlocaldefs

\begin{document}

\begin{frontmatter}
%%%%%%%%%%%%%%%%%%%%%%%%%%%%%%%%%%%%%%%%%%%%%%
%%                                          %%
%% Double-estimation-friendly inference for
%% high-dimensional misspecified models     %%
%%                                          %%
%%%%%%%%%%%%%%%%%%%%%%%%%%%%%%%%%%%%%%%%%%%%%%
\title{Double-estimation-friendly inference for high-dimensional misspecified models}

\runtitle{Double-estimation-friendly inference}
%\thankstext{T1}{A sample of additional note to the title.}

\begin{aug}
\author[A]{\fnms{Rajen D.} \snm{Shah}\ead[label=e1]{r.shah@statslab.cam.ac.uk}}
\and
\author[B]{\fnms{Peter} \snm{B{\"u}hlmann}\ead[label=e2]{buehlmann@stat.math.ethz.ch}}
%%%%%%%%%%%%%%%%%%%%%%%%%%%%%%%%%%%%%%%%%%%%%%
%% Addresses                                %%
%%%%%%%%%%%%%%%%%%%%%%%%%%%%%%%%%%%%%%%%%%%%%%
\address[A]{Statistical Laboratory, University of Cambridge, Cambridge, UK\printead{e1}.}
\address[B]{Seminar for Statistics,	ETH Zurich,	Zurich, Switzerland\printead{e2}.}
\end{aug}

\begin{abstract}
All models may be wrong---but that is not necessarily a problem for inference. Consider the standard $t$-test for  the significance of a variable $X$ for predicting response $Y$ whilst controlling for $p$ other covariates $Z$ in a random design linear model. This yields correct asymptotic type~I error control for the null hypothesis that $X$ is conditionally independent of $Y$ given $Z$ under an \emph{arbitrary} regression model of $Y$ on $(X, Z)$, provided that a linear regression model for $X$ on $Z$ holds. An analogous robustness to misspecification, which we term the ``double-estimation-friendly'' (DEF) property, also holds for Wald tests in generalised linear models, with some small modifications.

In this expository paper we explore this phenomenon, and propose methodology for high-dimensional regression settings that respects the DEF property. We advocate specifying (sparse) generalised linear regression models for both $Y$ and the covariate of interest $X$; our framework gives valid inference for the conditional independence null if either of these hold. In the special case where both specifications are linear, our proposal amounts to a small modification of the popular debiased Lasso test. We also investigate constructing confidence intervals for the regression coefficient of $X$ via inverting our tests; these have coverage guarantees even in partially linear models where the contribution of $Z$ to $Y$ can be arbitrary.
Numerical experiments demonstrate the effectiveness of the
methodology.
\end{abstract}

\begin{keyword}
\kwd{Conditional independence}
\kwd{High-dimensional inference}
\kwd{Debiased Lasso}
\kwd{Generalised linear models}
\kwd{Double robustness}
\end{keyword}

\end{frontmatter}
%%%%%%%%%%%%%%%%%%%%%%%%%%%%%%%%%%%%%%%%%%%%%%
%%%% Main text entry area:
\section{Introduction}
In this expository article, we describe a
concept of insensitivity or robustness against model misspecification in
linear and generalised linear models.
Our starting point is the observation that inference in a misspecified
linear model for the regression parameter still leads to correct
statements about certain conditional independencies if the relationships between the covariates takes an appropriate form. Our aim is to popularise this main
idea which, up to a few exceptions, seems to have been largely overlooked
in the statistical literature and textbooks; and also to further develop the
methodology and some theory for the case of high-dimensional linear and
generalised linear models.

\medskip\noindent
\textbf{Misspecified linear models and the $t$-test.}
We now describe a simple result (see Theorem~\ref{thm:lm})
which should serve as a motivation.
Consider data $(\mb Y, \mb X, \mb Z) \in \R^n \times \R^n \times
\R^{n\times p}$ (note $\mb X$ is a vector whilst $\mb Z$ is a matrix) for which we have postulated a random design linear model,
\begin{equation} \label{eq:lmY}
	\mb Y = \mb X \theta + \mb Z \beta^Y + \mbb\varepsilon,
\end{equation}
with $\mbb\varepsilon \sim \mathcal{N}_n(\mb 0, \sigma^2 \mb I)$ and design matrix $(\mb X, \mb Z)$ having i.i.d.\ Gaussian rows. The reason for distinguishing the covariate $\mb X$ from the other columns of $\mb Z$ is to focus attention on a single component of the vector of regression coefficients, namely $\theta$.
If this model is correctly specified, the $t$-statistic provides valid and optimal inference for $\theta$.

Now suppose that the model \eqref{eq:lmY} is misspecified and $\mb Y$ is a
nonlinear function 
of the Gaussian covariates and a (not necessarily Gaussian) error term. Then, the standard $t$-test
in the misspecified linear model for $\theta=0$ still provides asymptotically valid inference for testing the
null-hypothesis that $\mb Y$ is conditionally independent of $\mb X$ given all other
covariates $\mb Z$, in the  
sense that the type I error is asymptotically correctly controlled. In fact if $\mb Y =
\theta \mb X + f(\mb Z, \mbb\varepsilon)$ for an essentially arbitrary measurable function $f$, standard
confidence intervals for $\theta$ will be valid in this more general
partially linear model setting. This 
perhaps comes as a surprise! As we will explain, it is connected to the
fact that in the misspecified model, the projected parameter in the
specified linear model corresponding to $\mb X$ is exactly zero when we have the conditional independence $\mb Y \ci \mb X \,|\, \mb Z$; this in turn is a
consequence of the regression relation between $\mb X$ and $\mb Z$ being
linear due to the 
Gaussian assumption, that is we have $\E(\mb X \,|\, \mb Z) = \mb Z \beta^X$
for some $\beta^X \in \R^p$. 

\medskip
This is just a simple motivating example, and we will relax some
of the assumptions to provide a more general methodology and theory.
In particular, we show that this phenomenon also extends to generalised linear models (GLMs) in the
sense that if $\mb X \ci \mb Y \,|\, \mb Z$, then the estimated coefficient
corresponding to $\mb X$ following a generalised linear regression of $\mb
Y$ on $(\mb X, \mb Z)$ will have mean zero asymptotically if either the GLM
is valid, or if a linear regression model for $\mb X$ on $\mb Z$ holds (and
in the latter case, the GLM can be arbitrarily misspecified). 

Thus in general, basic statistical inference procedures concerning linear models and GLMs have validity
beyond the restrictive parametric settings for which they are
designed.
Our focus in this work is studying this robustness property for which we use the
term
\begin{quote}
	 DEF, for `\textbf{d}ouble-\textbf{e}stimation-\textbf{f}riendly'. The word ``double'' refers
	to the issue of specifying and estimating two models, and the double
	estimation leads then to more ``friendly'' results where valid inference is provided if either model is well-specified.
\end{quote}
%\begin{description}
%	\item DEF, for `\textbf{d}ouble-\textbf{e}stimation-\textbf{f}riendly'. The word ``double'' refers
%	to the issue of specifying and estimating two models, and the double
%	estimation leads then to more ``friendly'' results where valid inference is provided if either model is well-specified.
%\end{description}
%\Peter{Maybe make the DEF more ``visible. Perhaps as follows:\\
	%  Our focus in this work is studying this robustness property. We use the
	%  term
	%  \begin{description}
		%  \item DEF, for `double-estimation-friendly'. The word ``double'' refers
		%      to the issue of specifying and estimating two models, and the double
		%      estimation leads then to more ``friendly'' results.
		%    \end{description}
	With this term DEF we want to clearly distinguish it from double
	robustness, a concept whose relation to DEF is described below in Section
	\ref{sec:doublerobustness}.

	A substantial part of this work considers DEF methodology in
	high-dimensional regression where $p \gg n$.
	Driven by demands from a range
	of application areas, but perhaps most notably genomics, high-dimensional
	regression has received a great deal of attention over the last two
	decades; see for example the books
	\citet{buhlmann2011statistics,tibshirani2015statistical,wainwright_2019}
	and references therein. Whilst earlier work dealt primarily with point
	estimation of regression coefficients, more recently there has been a drive
	towards (Frequentist) uncertainty quantification, including testing for
	whether pre-specified regression coefficients are non-zero. Much of this
	work has centred on the so-called \emph{debiased Lasso}
	\citep{Zhang2013,vandegeer2014} which gives a construction of a coefficient
	estimate that unlike the more standard Lasso \citep{tibshirani96regression}
	on which it is based, is asymptotically unbiased and normally distributed;
	it can therefore serve as a basis for forming confidence intervals and
	hypothesis tests about the unknown true coefficient vector. 
	
	The debiased Lasso has been a major advance for inference in high-dimensional settings. However the validity of the statistical inferences it provides rests on the somewhat strong assumption that the true coefficient vector is highly sparse. For example, when testing whether $\mb Y = \mb Z \beta^Y + \mbb\varepsilon$, i.e.\ if the coefficient for $\mb X$ is $0$, guarantees for the debiased Lasso require that $s_Y := |\{j : \beta^{Y}_j \neq 0\}|$ satisfies $s_Y = o(\sqrt{n} / \log(p))$. 
	Given the preceding discussion, it is natural to ask whether the debiased Lasso is in some sense DEF. We show in this work that, with some small modifications, a version of the debiased Lasso has the DEF property. Specifically, a modified debiased Lasso gives a valid test for $\mb X \ci \mb Y \,|\, \mb Z$ if either the $X$-model, that is the model for $\mb X$ regressed on $\mb Z$, or the $Y$-model is a sparse linear model. Confidence intervals derived from the debiased Lasso however are not DEF and do rely heavily on a sparse linear $Y$-model. We demonstrate that confidence intervals constructed via inverting a DEF hypothesis test can lead to much better coverage properties. Whilst not part of the main focus of this work, we also show how a related approach may be used to construct confidence intervals for $w^T\beta^Y$, where $w \in \R^p$ is a possibly dense contrast vector.
	
	In many settings, for example when $\mb X$ is binary, a linear model for $\mb X$ on $\mb Z$ seems unlikely to hold. It would therefore be desirable to have a DEF procedure for testing the conditional independence relationship $\mb X \ci \mb Y \,|\, \mb Z$ that is valid when either the $Y$-model or the $X$-model are sparse generalised linear models. For example when both $\mb Y$ and $\mb X$ are binary we might wish to specify both models as logistic regression models. By first adapting our proposed DEF procedure to settings with linear $X$- and $Y$-models with heteroscedastic errors, we show how generalised linear models can be handled within our DEF methodology.
	
	Below we mention some related work. We first discuss how our DEF concept and methodology relates to the literature on double robustness, and then look at other work in high-dimensional inference that bears some relation to ours here.
	
	\subsection{Relation to double robustness} \label{sec:doublerobustness}
	The concept of double robustness has been developed in the context of missing values and
	causal effects; the latter can be seen as a missing value problem with
	unobserved potential outcomes. One specifies a model for the
	response and a model for the missingness (e.g., unobserved potential outcome),
	both as a function of covariates. The double robustness property is then (typically) as follows: if only one of the models is correctly specified, one
	can still obtain consistent estimates of average effects. This conclusion comes as a result of the bias of a doubly robust estimator taking the form of a product of estimation errors relating to each of the aforementioned models. In order for the product to tend to zero, only one of the terms in the product need tend to zero; we refer to
	\citet{robins1995semiparametric,scharfstein1999adjusting,kang2007demystifying,cao2009improving,rotnitzky2012improved},
	among many other contributions in the literature. 
	
	Whilst the philosophy of DEF is similar to that of double robustness in that it aims to ``give the analyst two
	chances, instead of only one, to make a valid inference''
	\citep{bangrobins2005}, there are several differences.
	Firstly, we are asking for valid inferential procedures, i.e.\ hypothesis tests and confidence intervals, when either the $X$-model or the $Y$-model is misspecified. Whereas for consistency, it suffices for one of the terms composing the bias to go to zero, for our purposes this would need to vanish at a rate dominated by the variance which is typically $n^{-1/2}$. The requirement that the product of estimation error rates bounding the bias goes zero faster than $n^{-1/2}$ has been referred to as \emph{rate double robustness} \citep{2019arXiv190403737S}. However directly applying known estimation error rates for high-dimensional regression to achieve rate double robustness 
	gives rise to procedures for hypothesis testing that require both the $X$ and $Y$-models to be sparse regression models with sparsity levels $s_X, s_Y = o(\sqrt{n}/\log(p))$ \citep{chernozhukov2016double, shah2018hardness, dukes2018high}; a stronger requirement than needed for the debiased Lasso, which only assumes a sparse $Y$-model, and stronger still than our DEF methodology, which requires either a well-specified sparse $Y$-model or $X$-model.
	
	In parallel work to ours, \citet{bradic2019sparsity} introduce the concept of \emph{sparsity double robustness} in the context of estimation of average treatment effects that refers to a weakening of the strong sparsity conditions imposed by rate double robustness above; however in contrast to our DEF principle, this still requires sparse $X$ and $Y$-models.
	
	A second difference is that whereas doubly robust methods are typically
	semiparametrically efficient as they are often derived by considering
	efficient influence functions for the parameters at hand, this sort of efficiency
	does not necessarily arise in the more general settings covered by our idea
	of DEF inference.
	Because of these differences we use the new terminology to distinguish
	the concept from double robustness.
	
	\subsection{Other related work} \label{sec:existing}
	In the low-dimensional setting, early work on single-index models \citep{MR689741,MR1015136,MR1105834} has shown that OLS regression on Gaussian covariates can correctly estimate the direction of the vector of regression coefficients up to an unknown sign. This property is somewhat related to the DEF property of OLS, though deals with a rather specific form of misspecification of a linear model. 
	
	The concept of leveraging an $X$-model in assessing the contribution of a covariate $\mb X$ to a response $\mb Y$ whilst controlling for additional covariates $\mb Z$ has a long history, and the modelling of propensity scores when estimating average treatment effects is one example of this \citep{rosenbaum1983central}. The work of \citet{robins1992estimating} proposes to exclusively estimate an $X$-model in more general settings, and this idea has also appeared more recently in the model-$X$ knockoff framework \citep{candes2018panning}. The conceptual difference though is that with
	DEF (and also double robustness as discussed above), both the $X$-model and $Y$-model are estimated but one does not need to know
	which of the two models is correct.
	
	Some recent work has looked at DEF procedures for different
	high-dimensional settings. \citet{shah2018goodness} studied a certain
	regularised partial correlation proposed in \citet{ren2015asymptotic}; the
	latter work shows this test statistic is valid for testing $\mb X \ci \mb Y
	\,|\, \mb Z$ when both the $X$-model and $Y$-models are sparse linear models,
	whilst the former shows in fact only the $Y$-model needs to be true for
	correct type I error control. As the test statistic is symmetric in $\mb X$
	and $\mb Y$, we can further conclude it has the DEF property. Our proposed DEF methodology
	for the high-dimensional setting builds on this work, generalising it to
	allow for generalised linear $X$ and $Y$-models. This approach is not the
	only possibility for DEF methodology in the high-dimensional setting, and
	\citet{zhu2018significance} look at another similar test statistic they call CorrT that delivers hypothesis tests with asymptotic type I error control in the setting
	where the $Y$-model is permitted to be a dense linear model, whilst the
	$X$-model must be a sparse linear model. Again, this test statistic has a
	DEF-like property as a consequence of its symmetry, though the dense linear model still entails some restrictions on the model class, see the discussion following Theorem~\ref{thm:DEF_lin} in Section~\ref{sec:high_lin}.
	
	\citet{buhlmann2015high} consider inference with
	the debiased Lasso in misspecified linear models, but where the best linear predictor of the response given covariates, is sparse, and the $X$-model is linear. This is related to our results and methodology here, though in contrast we aim for valid inference with no sparsity requirements on one of either the $X$ or $Y$-models.
	%   : their Proposition 3
	%  says that when the $X$-model is linear, one can still control the type I
	%  error of the debiased Lasso for testing. This is exactly as with
	%  DEF. Empirically we find that our DEF statistic though has a slightly
	%  better performance.
	%\Rajen{Your paper with
		%  Sara?} \Peter{B\"uhlmann and van de Geer (2015) considered inference with
		%  the debiased Lasso in misspecified linear models: their Proposition 3
		%  says that when the $X$-model is linear, one can still control the type I
		%  error of the debiased Lasso for testing. This is exactly as with
		%  DEF. Empirically we find that our DEF statistic though has a slightly
		%  better performance.\\
		%  Does this make sense?} 
	We note that our work also connects more generally to a thriving literature on high-dimensional inference. We refer to \citet{dezeure2015high} for a review of some of the most important developments that are related to our work here.
	
	\subsection{Organisation of the paper}
	The rest of the paper is organised as follows. In Section~\ref{sec:low} we study the low-dimensional setting and formally set out the DEF properties of standard inference procedures for linear and generalised linear models. We then turn to the high-dimensional setting and study in Section~\ref{sec:high_lin} the case where we allow either the regression model for $\mb Y$ on $\mb Z$ or that for $\mb X$ on $\mb Z$ to be linear. In Section~\ref{sec:high_conf} we detail the construction of confidence intervals in partially linear high-dimensional models using the classical duality between confidence regions and hypothesis tests. We then study the setting where the models of $\mb Y$ and $\mb X$ are generalised linear models. Some numerical experiments are presented in Section~\ref{sec:experiments} and we conclude with a discussion in Section~\ref{sec:discuss}. The appendix contains proofs omitted in the main text, a construction for confidence regions for $w^T\beta^Y$ based on the methodology set out in Section~\ref{sec:high_lin}, a description of how square-root Lasso solutions may be computed given regular Lasso solutions, and some additional numerical experiments.
	
	%\subsection{Notation}
	
	\section{Low dimensions} \label{sec:low}
	Recall that $(\mb Y, \mb X, \mb Z) \in \R^n \times \R^n \times \R^{n\times p}$ and we are interested in the relationship between $\mb Y$ and $\mb X$, and specifically testing the conditional independence $\mb X \ci \mb Y \,|\, \mb Z$. We first study the DEF property of the standard $t$-statistic in the linear model, before turning to generalised linear models in Section~\ref{sec:low_gen}.
	\subsection{Linear models} \label{sec:low_lin}
	%We also define (X1) to be the equivalent of (Y1) but with $\mb Y$ replaced everywhere by $\mb X$.
	Let $\tilde{\mb Z} := (\mb X \,, \mb Z) \in \R^{n \times (p+1)}$ and let
	$(\hat{\theta}, \hat{\beta}^Y) \in \R \times \R^p$ be the regression
	coefficient vector from an OLS regression of $\mb Y$ on $\tilde{\mb
		Z}$. Further let $\tilde{\mb P}$ and $\mb P$ be the orthogonal
	projections on to $\tilde{\mb Z}$ and $\mb Z$ respectively. Also define
	$\tilde{\sigma}^2 = \|\mb Y - \tilde{\mb P} \mb
	Y\|_2^2/(n-p-1)$. The usual
	$t$-statistic for testing the significance of variable $\mb X$ is given by $T_{\OLS} := \hat{\theta} / \sqrt{\{(\tilde{\mb Z}^T \tilde{\mb Z})^{-1}\}_{11} \tilde{\sigma}^2}$. Denote by $\mb R:= (\mb I - \mb P)\mb X$
	the residuals from regressing $\mb X$ on $\mb Z$.
	
	Consider the following set of assumptions.
	%The particular form of linear model we assume here is formalised in (Y1) below.
	%\begin{itemize}
	%\item[(Y1)] We have $\mb Y = \mb Z \mbb\beta^Y + \mbb\varepsilon$ with $\E(\varepsilon_i \,|\, \mb Z) =0$, $\E(\varepsilon_i^2 \,|\, \mb Z) = \sigma^2>0$, $\E(\varepsilon_i^4 \,|\, \mb Z) < M$ for some constant $M$, and the $\varepsilon_i$ are independent conditional on $\mb Z$.
	%\end{itemize}
	\begin{itemize}
		%\item[(A1)] $\mb Z^T \mb Z / n \inprob \Sigma$ where $\Sigma$ is positive definite and moreover $\E (\mb Z^T \mb Z / n) \to \Sigma$.
		\item[(Y1)] We have $\mb Y = \mb Z \beta^Y + \mbb\varepsilon$ with $\E(\varepsilon_i \,|\, \mb Z) =0$, $\E(\varepsilon_i^2 \,|\, \mb Z) = \sigma^2>0$, $\E(|\varepsilon_i|^{2+\delta} \,|\, \mb Z) < M$ for some constants $M,\delta, \sigma^2 >0$, and the $\varepsilon_i$ are independent conditional on $\mb Z$.
		\item[(Y2)] We have $\pr(\mb R = \mb 0) \to 0$ and for some $\delta > 0$,
		\begin{equation} \label{eq:R}
			A_n := \begin{cases}
				\frac{1}{\|\mb R\|_2^{2+\delta}} \sum_{i=1}^n |R_{i}|^{2+\delta} & \quad \text{ if } \mb R \neq \mb 0, \\
				0 & \quad \text{ if } \mb R = \mb 0,
			\end{cases} %\qquad \text{satisfies} \qquad A_n \inprob 0.
		\end{equation}
	satisfies $A_n \inprob 0$.
	\end{itemize}
	Condition (Y1) formalises the particular form of the linear model we assume
	here (under the null-hypothesis), which includes the normal linear model, for example, but is rather more general.
	Condition (Y2) enforces that no individual residual is too extreme. Indeed, it is sufficient that $\max_i R_i / \|\mb R\|_2 \inprob 0$. This would typically be satisfied if the rows of $(\mb X, \mb Z)$ were i.i.d.\ for example, but is much weaker.
	%Indeed we could typically expect the numerator and denominator in the expectation to both be of constant order, so the pre-factor of $1/\sqrt{n}$ should easily drive the quantity to 0.
	%Let (X1) and (X2) be the equivalent of (Y1) and (Y2) above but with $\mb X$
	%replaced with $\mb Y$ and vice versa.
	We also introduce the following.
	\begin{itemize}
		\item[(X$j$)] The equivalent of (Y$j$) above but with $\mb X$
		replaced with $\mb Y$ and vice versa, for $j \in \{1,2\}$.
	\end{itemize}
	The Theorem below shows that $T_{\OLS}$ has a DEF property.
	\begin{thm} \label{thm:lm}
		%Suppose assumption (A1) holds.
		Suppose $p/n \to 0$.
		If either (X1) and (X2) or (Y1) and (Y2) hold, then under the null hypothesis that $\mb X \ci \mb Y \,|\, \mb Z$, we have $T_{\OLS} \indist \mathcal{N}(0, 1)$.
	\end{thm}
	The result may be viewed as a consequence of the close relationship between the $t$-statistic above and the partial correlation
	\[
	\hat{\rho} := \frac{\mb X^T (\mb I - \mb P) \mb Y}{\|(\mb I - \mb P)\mb X\|_2 \|(\mb I - \mb P)\mb Y\|_2}.
	\]
	This can also be interpreted as a test statistic based on a score test for $\theta=0$ when it is assumed the errors are Gaussian.
	%Indeed
	One can verify that
	\begin{equation} \label{eq:parcor_rel}
		T_{\OLS} = \sqrt{n-p-1} \frac{\hat{\rho}}{\sqrt{1-\hat{\rho}^2}},
	\end{equation}
	so the distributional result for $T_{\OLS}$ follows from $\sqrt{n}\hat{\rho} \indist \mathcal{N}(0, 1)$. As $\hat{\rho}$ is symmetric in $\mb X$ and $\mb Y$ it is unsurprising that this has a DEF property.
	Indeed, the DEF approach suited to the high-dimensional setting we present in Section~\ref{sec:high}, is based on a certain regularised partial correlation.
	
	%\subsubsection{Confidence intervals} \label{sec:conf}
	%Theorem~\ref{thm:lm} also shows that the standard confidence interval
	%\[
	%a
	%\]
	
	%\RS{To be moved to an apendix?}
	We also remark that under the assumption that $\mb Y = \mb Z \beta^Y + \mbb\varepsilon$ with $\mbb\varepsilon \sim \mathcal{N}(\mb 0 , \sigma^2 \mb I)$, we have the exact distributional relationship
	\[
	\hat{\rho}\sqrt{\frac{n-p-1}{1-\hat{\rho}^2}} \sim t_{n-p-1}.
	\]
	The symmetry of this statistic in $\mb X$ and $\mb Y$ means that the distributional result also holds when an analogous normal linear model for $\mb X$ on $\mb Z$ holds. This may be used to yield a DEF test for conditional independence with exact type I error control in finite samples, under these additional Gaussianity assumptions.
	
	\begin{ex} \label{ex:lm}
		The famous diabetes dataset of \citet{efron04least} contains $p=10$ predictors (age, sex, BMI, etc.) measured for $n=442$ patients. We take these covariates as our matrix $\mb Z \in \R^{n \times p}$ and generate an additional predictor $\mb X \in \R^n$ with entries $X_i = \sum_j Z_{ij} + \varepsilon^X_i$ where $\varepsilon^X_i + 1 \iid \text{Exp}(1)$. We ignore the original response of the design matrix and generate a new response $\mb Y \in \R^n$ that depends nonlinearly on $\mb Z$ through $Y_i =\eta_i \zeta_i$ where $\zeta_i \iid \chi^2_1$ and
		\begin{align}
			\eta_i = \sum_{j,k} \frac{\exp(Z_{ij}Z_{ik})}{1 + \exp(Z_{ij}Z_{ik})}. \label{eq:nonlin_pred}
		\end{align}
		In this setup we then have $Y_i \ci X_i \,|\, Z_i$ and the $X$-model is a
		linear regression model. Theorem~\ref{thm:lm} suggests that the
		$t$-statistic $T_{\OLS}$ corresponding to $\mb X$ should have a
		distribution well-approximated by a standard normal. The left panel of
		Figure~\ref{fig:lm} plots the histogram of $T_{\OLS}$ computed on 500
		simulated datasets generated through the construction above. We do indeed
		see a close agreement with a standard normal density, verifying the
		theoretical result. The right panel plots the coefficient estimate
		$\hat{\theta}$ corresponding to $\mb X$ when the equation for $\mb Y$ has
		$\mb X$ added (i.e., the null-hypothesis does not hold). It is easy to see that compared to the previous setup, this coefficient will be shifted by $1$, and hence asymptotically should have a Gaussian distribution centred on $1$, as we observe in the plot.
		\begin{figure*}[h]
			\centering
			\includegraphics[width=\textwidth]{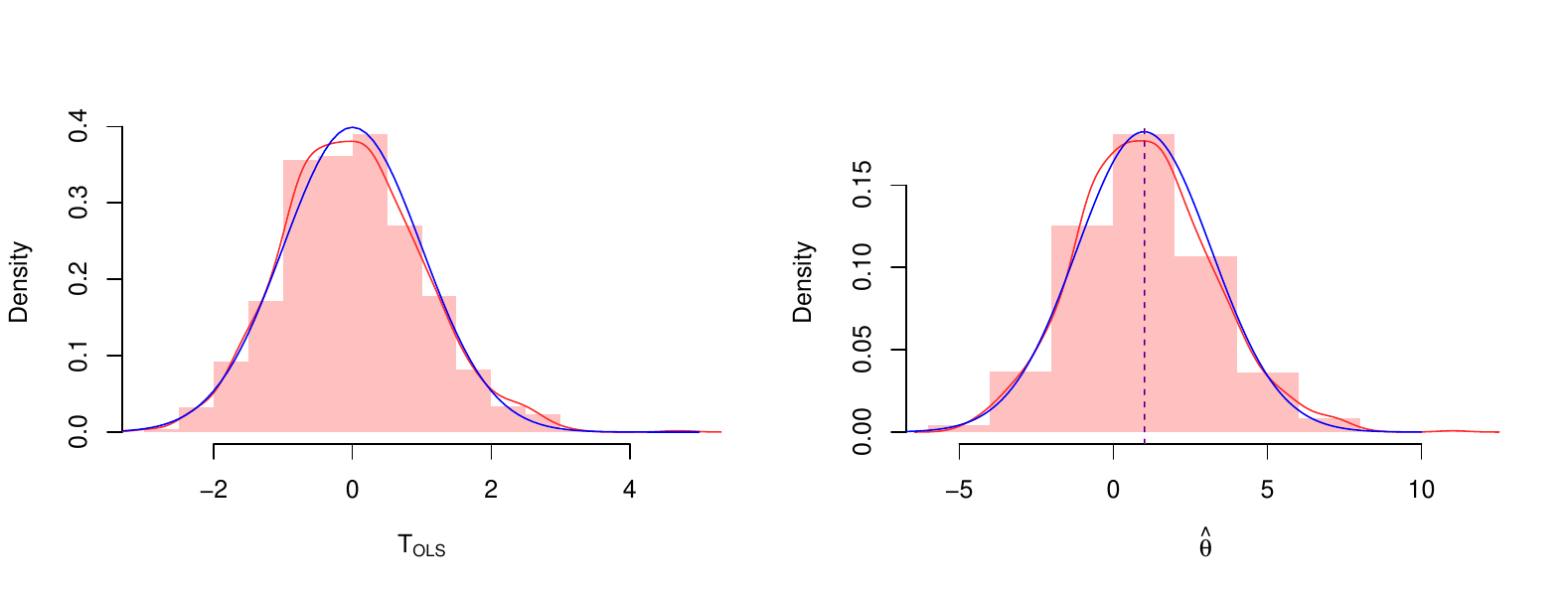}
			\caption{Histograms of $T_{\OLS}$ (left plot) and $\hat{\theta}$ (right plot) for the setup described in Example~\ref{ex:lm}. The red curves are kernel density estimates. We see close agreement with the theoretical normal density (blue curves). The vertical dashed red lines and blue lines in the right plot are the empirical and theoretical means respectively; their proximity in this example makes them hard distinguish visually.\label{fig:lm}}
		\end{figure*}
	\end{ex}
	
	\subsection{Generalised linear models} \label{sec:low_gen}
	It is well known that maximum likelihood estimators under misspecification are, given regularity conditions, asymptotically normal about a parameter vector corresponding to the model closest to the ground truth in terms of Kullback--Leibler divergence \citep{huber1967behavior,white1982maximum}. This fact is typically used as reassurance that whilst all statistical models are wrong, provided one is working with a model that is a good enough approximation to the truth, maximum likelihood estimation is nevertheless useful. 
	However, as we shall explain, in terms of conditional independence testing, maximum likelihood estimation of generalised linear models can form the basis of a valid test even under severe misspecification.
	
	In this section we will assume that the rows $(X_i, Y_i, Z_i)$ of $(\mb X, \mb Y, \mb Z) \in \R^{n \times (2+p)}$ are independent copies of the random triple $(X, Y, Z) \in \R \times \R \times \R^p$. Consider a generalised linear model relating response vector $\mb Y$ to covariates $(\mb X, \mb Z)$, or more generally, a model where the density $f_{Y|X, Z}$ of $Y$ conditional on $(X, Z)$ (with respect to a measure $\mu$) takes the form
	\begin{equation} \label{eq:log_lik}
		f_{Y|X, Z}(y|x, z) = L(x\theta + z^T\beta^Y; y)
	\end{equation}
	for $(\theta, \beta^Y) \in \Theta \subseteq \R^{p+1}$.
	We will assume that $L$ is twice differentiable in its first argument.
	Define $\ell := \log L$ and $U := \ell'$ where the prime denotes a derivative with respect to the first argument; we will typically suppress the dependence of $U$ on its second argument $y$ for simplicity.
	Under regularity conditions, the maximum likelihood estimator
	%$(\hat{\theta}^Y, \hat{\beta}^Y) \in \R\times \R^p$
	\[
	(\hat{\theta}, \hat{\beta}^Y) := \argmin_{(t, \beta) \in \Theta} -\sum_{i=1}^n \ell(X_i t + Z_i^T\beta; Y_i)
	\]
	is asymptotically normal centred on $(\theta^*, \beta^*)$, which solve for $(t, \beta) \in \Theta$ the score equations
	\begin{align}
		\E \{X \,U(X t + Z^T\beta)\} &=0 \label{eq:x_score} \\
		\E\{Z \,U(X t + Z^T\beta)\} &=0. \label{eq:z_score}
	\end{align}
	When \eqref{eq:log_lik} holds (which includes as a special case when a generalised linear model
	is correct), under regularity conditions, we will have $(\theta^*, \beta^*) = (\theta, \beta^Y)$.
	In order for inference based on $\hat{\theta}$ to provide useful
	information concerning the conditional independence $X \ci Y \,|\,Z$ when
	$\eqref{eq:log_lik}$ does not hold, we would like $\theta^* = 0$ in the
	case of conditional independence. Analogously to the case with linear
	models discussed in the previous section, we have that  regardless of the
	form of the $Y$-model, provided the $X$-model is
	linear, it holds that $\theta^*=0$; here though we additionally require that the solution to \eqref{eq:x_score} and \eqref{eq:z_score} is unique to derive this conclusion.
	\begin{thm} \label{thm:gen_lin}
		Suppose $X \ci Y \,|\,Z$.
		Let $\beta^\dagger \in \R^p$ maximise the expected log-likelihood $\E\ell(Z^T\beta;Y)$ over $\beta$. Assume regularity conditions set out in Section~\ref{sec:reg1} of the appendix.
		%Suppose that $\beta^* \in \R^p$ satisfies the $p$ equations \eqref{eq:z_score} for $\beta$ with $\theta =0$ when $X \ci Y \,|\,Z$.
		Suppose that either the $Y$-model is well-specified so \eqref{eq:log_lik} holds, or the $X$-model is linear so $\E(X\,|\,Z) = Z^T\beta^X$. Then $(t, \beta) = (0, \beta^\dagger)$ satisfies the score equations \eqref{eq:x_score}, \eqref{eq:z_score}.
	\end{thm} 
	Theorem~\ref{thm:gen_lin} shows that under the $X$-model, the parameter corresponding to the projection of the truth on to the purported $Y$-model is $0$ under conditional independence.
	A standard Wald test for whether $\theta=0$ will however not be valid under general misspecification as the asymptotic variance of $\hat{\theta}$ will not necessarily be given by the $(1, 1)$ entry of the inverse Fisher information matrix for $(\theta, \beta^Y)$. Indeed, it is well-known that, under regularity conditions, the variance of $\hat{\theta}$ is given by the sandwich formula
	\begin{equation} \label{eq:asymp_norm}
		\sqrt{n}\left(\begin{pmatrix} \hat{\theta} \\ \hat{\beta}^Y \end{pmatrix} - \begin{pmatrix}\theta^* \\ \beta^*\end{pmatrix} \right) \indist \mathcal{N}\left(0, H^{-1} V H^{-1}\right),
		%\Cov\left(\frac{\partial \ell(X\theta+Z^T\beta)}{\partial \theta} \bigg|_{\theta=\theta^*}, \frac{\partial \ell(X\theta+Z^T\beta)}{\partial \beta} \bigg|_{\beta=\beta^*}\right) 
	\end{equation}
	where $V$ is the covariance matrix of the derivative of $\ell(X\theta +
	Z^T\beta;Y)$ with respect to $(\theta, \beta)$ evaluated at $(\theta^*,
	\beta^*)$ (satisfying the score equations \eqref{eq:x_score}, \eqref{eq:z_score}) and $H$ is the negative expectation of the corresponding Hessian matrix:
	\begin{align*}
		V &:= \E\left(\begin{pmatrix} X \\ Z\end{pmatrix} \begin{pmatrix} X \\ Z\end{pmatrix}^T U^2(X\theta^* + Z^T\beta^*) \right) \\
		H &:= -\E\left(\begin{pmatrix} X \\ Z\end{pmatrix} \begin{pmatrix} X \\ Z\end{pmatrix}^T U'(X\theta^* + Z^T\beta^*) \right).
	\end{align*}
	The matrices $V$ and $H$ may be estimated individually using the data via
	several methods \citep{mackinnon1985some}. However, if either the $X$-model is a homoscedastic linear model, or the $Y$-model holds, some simplifications are possible, as the result below describes.
	\begin{thm} \label{thm:gen_var}
		Suppose $X \ci Y \,|\,Z$ and assume regularity conditions set out in Section~\ref{sec:reg2} of the appendix.
		Suppose either \eqref{eq:log_lik} holds with $U = \ell'$, or $\E (X\,|\,Z) = Z^T\beta^X$. We additionally assume $\Var(X\,|\,Z)=\Var(X)$ in the latter case.
		Then we have
		\[
		(H^{-1} V H^{-1})_{11} = - (H^{-1})_{11} \frac{\E\{U^2(Z^T\beta^*)\}}{\E\{ U'(Z^T\beta^*)\}}.
		\]
	\end{thm}
	The correction factor for the usual inverse of the Fisher information may be readily estimated by
	\begin{equation} \label{eq:correction}
		\hat{C}_1 := - \frac{\sum_{i=1}^nU^2(Z_i^T\hat{\beta}^Y)}{\sum_{i=1}^n U'(Z_i^T\hat{\beta}^Y)},
	\end{equation}
	or indeed a variant of the above with $Z_i^T\hat{\beta}^Y$ replaced
	everywhere by $X_i\hat{\theta} + Z_i^T\hat{\beta}^Y$ which we will refer to as $\hat{C}_2$. Writing $\hat{H}$ for the empirical version of $H$,
	\[
	\hat{H} := -\frac{1}{n} \sum_{i=1}^n \begin{pmatrix} X_i \\ Z_i\end{pmatrix} \begin{pmatrix} X_i \\ Z_i\end{pmatrix}^T U'(X_i\hat{\theta} + Z_i^T\hat{\beta}^Y),
	\]
	we may define for $j=1,2$, the test statistics
	\[
	T_{\GLM,j} := \frac{\sqrt{n}\hat{\theta}}{\sqrt{\hat{C}_j (\hat{H}^{-1})_{11}}}.
	\]
	Putting together Theorems~\ref{thm:gen_lin} and \ref{thm:gen_var} we have the following result.
	\begin{thm} \label{thm:wald}
		Suppose $X \ci Y \,|\, Z$ and \eqref{eq:asymp_norm} holds where $(\theta^*,\beta^*)$ is the unique solution in $(t,\beta)$ to \eqref{eq:x_score} and \eqref{eq:z_score}. Assume that $\hat{H} \inprob H$ with $H$ positive definite and assume the regularity conditions set out in Section~\ref{sec:reg3}. Suppose that either the $Y$-model is well-specified so \eqref{eq:log_lik} holds, or the $X$-model is linear so $\E(X\,|\,Z) = Z^T\beta^X$. Then for $j=1,2$ we have
		\[
		T_{\GLM,j} \indist \mathcal{N}(0, 1).
		\]
	\end{thm}
	\begin{ex} \label{ex:pois}
		We use a similar setup to Example~\ref{ex:lm} but here generate the response $\mb Y \in \R^n$ according to $Y_i \iid \text{Poisson}(\mu_i)$ with
		\[
		\log(\mu_i) = a_1 \sum_j Z_{ij} + \sigma a_2 \eta_i
		\]
		with $\sigma \in \{0, 2, 4\}$ and factors $a_1$ and $a_2$ chosen so the maximum absolute value over $i$ of the two terms above is $3$ to ensure $\E Y_i$ does not take values that are too large. We consider testing the significance of the variable $\mb X$ using (a) standard Wald-based $p$-values assuming a Poisson log-linear model, (b) the equivalent using a quasi-Poisson likelihood and (c) using $T_{\GLM,2}$.
		%the correction factor \eqref{eq:correction} with $Z_i^T\hat{\beta}^Y$ replaced everywhere by $X_i\hat{\theta} + Z_i^T\hat{\beta}^Y$.
		We plot in Figure~\ref{fig:pois} the empirical distribution functions of the $p$-values observed over $500$ replicates of the three settings determined by $\sigma$. As expected, for the well-specified case with $\sigma=0$ all $p$-values are roughly uniformly distributed. However for increasing levels of misspecification, the standard $p$-values (a) tend to be more anti-conservative, a phenomenon which occurs to a lesser extent for the quasi-likelihood-based $p$-values (b). The correction factor (c) ensures that $p$-values corresponding to $T_{\GLM,2}$ are  approximately uniform across all of the settings considered.
		\begin{figure*}[h]
			\centering
			\includegraphics[width=\textwidth]{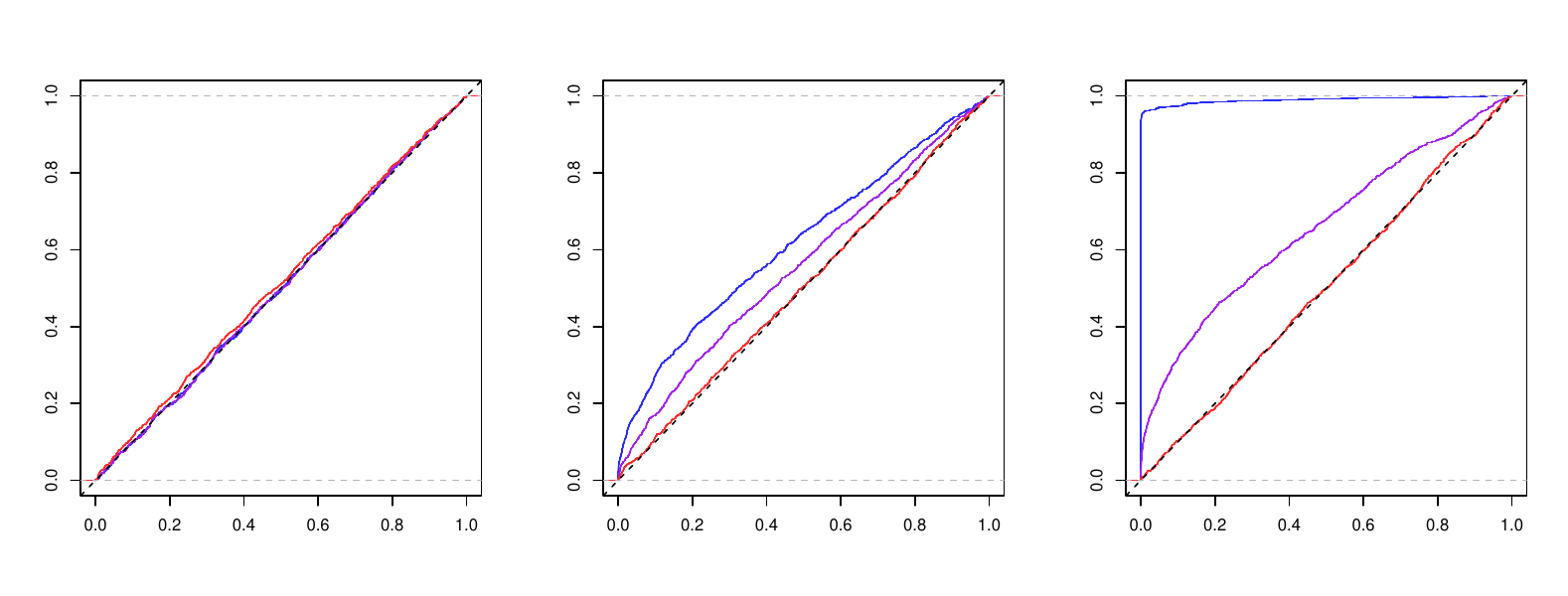}
			\caption{Empirical distribution functions of $p$-values from the simulation setups of Example~\ref{ex:pois} with $\sigma=0, 2, 4$ from left to right. Blue, purple and red curves correspond to naive $p$-values (a), quasi-likelihood-based $p$-values (b) and $p$-values based on $T_{\GLM,2}$ 
				%employing the correction factor suggested by Theorem~\ref{thm:gen_var} \eqref{eq:correction}
				(c), respectively. Type I errors of the resulting tests are well-controlled for (c), but (a) and (b) fail to maintain nominal levels under misspecification.\label{fig:pois}}
		\end{figure*}
	\end{ex}

	\section{High dimensions} \label{sec:high}
	We have seen in the previous section how classical linear and generalised
	linear model inferential tools have the DEF property. In the case of linear
	models, this could be deduced from the similarity of the standard
	$t$-statistic to partial correlation. For generalised linear models, the
	DEF property is perhaps more surprising. Our analysis first used the fact
	that maximum likelihood converges to a projection of the ground truth, and
	then considered the projected parameters themselves. There is however no
	analogue of the classical Huber--White results on the properties of
	maximum likelihood in nonlinear models under misspecification
	available for high-dimensional 
	estimators. Our approach to DEF inference in high-dimensional settings will
	therefore be based around versions of partial correlation. We first study
	linear models before turning to the case of high-dimensional generalised
	linear models.
	
	\subsection{Linear models} \label{sec:high_lin}
	One of the most popular methods for testing the significance of predictors
	in high-dimensional regression problems is the so-called debiased Lasso
	\citep{Zhang2013}. We begin by discussing this approach, in order to motivate our DEF methodology.
	
	The debiased Lasso works as follows: first we form estimates $(\hat{\theta}, \check{\beta}^Y)$ through a Lasso regression of $\mb Y$ on $(\mb X, \mb Z)$, and also conduct a Lasso regression of $\mb X$ on $\mb Z$ to give a coefficient estimate $\hat{\beta}^X$. There are a variety of choices of tuning parameters for each of these regressions; to ensure that these tuning parameters do not depend on the noise variances of the respective regressions, we may use a particular parametrisation of the Lasso known as the square-root Lasso regressions \citep{Belloni2011, sun2012scaled}:
	\begin{align}
		(\hat{\theta}, \check{\beta}^Y) &:= \argmin_{(t, \beta) \in \R^{1 + p}} \{\|\mb Y - \mb X t - \mb Z \beta\|_2/\sqrt{n} + \lambda_Y\|\beta\|_1\},\label{eq:DB_lasso}\\
		\hat{\beta}^X &:= \argmin_{\beta \in \R^p} \{\|\mb X - \mb Z \beta\|_2/\sqrt{n} + \lambda_X \|\beta\|_1\}. \notag
	\end{align}
	Here we may take $\lambda_X=\lambda_Y = A\sqrt{2\log(p) / n}$ for $A>1$. Note that we have denoted the estimate of the coefficient vector for $X$ as $\check{\beta}^Y$ in order to distinguish it from $\hat{\beta}^Y$ introduced in \eqref{eq:beta_hat_DEF} below.
	The square-root Lasso may be computed easily using standard software that computes regular Lasso solutions: see Section~\ref{sec:square-root_Lasso} in the appendix.
	
	We then construct a test statistic $T_{\DB}$ for assessing the conditional independence $\mb X \ci \mb Y \,|\, \mb Z$ as follows:
	\begin{equation*}
		T_{\DB} := \sqrt{n} \frac{(\mb Y - \mb Z\check{\beta}^Y)^T(\mb X - \mb Z \hat{\beta}^X)}{\|\mb Y  - \hat{\theta} \mb X - \mb Z \check{\beta}^Y\|_2 \|\mb X  - \mb Z \hat{\beta}^X\|_2}.
	\end{equation*} 
	When the $Y$-model is a sparse linear model so $\mb Y = \mb Z \beta^Y + \mbb\varepsilon$ with $\beta^Y$ sparse and $\mbb\varepsilon \sim \mathcal{N}_n(\mb 0, \sigma^2 \mb I)$, we have that $T \indist \mathcal{N}(0, 1)$ as we now outline.
	Let us write
	\begin{align*}
		\mb R &:= \mb X - \mb Z\hat{\beta}^X, \\
		\hat{\sigma} & := \|\mb Y  - \hat{\theta} \mb X - \mb Z \check{\beta}^Y\|_2 / \sqrt{n}.
	\end{align*}
	A consequence of the stationarity conditions (the so-called KKT conditions) for the optimisation problem defining $\hat{\beta}^X$ is that, provided $\mb R \neq \mb 0$,
	\begin{equation} \label{eq:KKT_X}
		\frac{1}{\sqrt{n}} \|\mb Z^T \mb R \|_\infty / \|\mb R\|_2 \leq \lambda_X.
	\end{equation}
	We may thus decompose $T_{\DB}$ as follows:
	\begin{equation*}% \label{eq:T_DB}
		T_{\DB} = \frac{1}{\hat{\sigma}}\frac{\mb R^T}{\|\mb R\|_2} \mbb\varepsilon  + \frac{1}{\hat{\sigma}} (\beta^Y-\check{\beta}^Y)^T\mb Z^T \frac{\mb R}{\|\mb R\|_2}=: \text{(i)}+\text{(ii)}.
	\end{equation*}
	Conditioning on $\mb R$, $\mb R^T \mbb\varepsilon /\|\mb R\|_2$ is a weighted sum of the independent and identically distributed $\varepsilon_i$, and thus will have an asymptotic Gaussian distribution under weak conditions on $\mb R$; in fact if the $\varepsilon_i$ are Gaussian themselves we will have $\mb R^T \mbb\varepsilon /\|\mb R\|_2 \,|\, \mb R \sim \mathcal{N}(0, \sigma^2)$ exactly, and of course the unconditional distribution will hence also be Gaussian. If $\hat{\sigma} \inprob \sigma$, then by Slutsky's Lemma we will have that (i) converges in distribution to a standard normal. In order to guarantee this, we may appeal to known results about the square-root Lasso \citep{sun2012scaled}. These rest on a compatibility factor $\phi^2$ \citep{BuhlmannGeer2009} being bounded away from zero:
	\begin{equation} \label{eq:compat}
		\phi^2 := \inf_{\substack{(t, \beta)\in \R^{1+p} \\ |t|+\|\beta_{S_Y^c}\|_1 \leq 3\|\beta_{S_Y}\|_1 \neq 0}} \frac{\|\mb X t + \mb Z\beta\|/n}{\|\beta_{S_Y}\|_1 / s_Y};
	\end{equation}
	here $S_Y := \{j:\beta^Y_j \neq 0\}$, $s_Y := |S_Y|$ and we have used the notation that for any vector $b \in \R^p$ and set $S \subseteq \{1,\ldots,p\}$, $b_S \in \R^{|S|}$ is the subvector of $b$ composed of those components of $b$ indexed by $S$. Roughly speaking, designs with large compatibility factors cannot have very highly correlated columns. Provided $\phi^2 \gtrsim 1$, we have $\hat{\sigma} \inprob \sigma$  and also $\|\check{\beta}^Y - \beta^Y\|_1 \lesssim s_Y \sqrt{\log(p)/n}$ with high probability, when $\lambda_Y \asymp \sqrt{\log(p)/n}$ \citep{van2016estimation}. This second property may be used to bound (ii) via
	%Provided a so-called compatibility constant \citep{BuhlmannGeer2009} that 
	%\Rajen{Perhaps for statistical science we need a bit more background? E.g.\}
		%Under reasonable conditions (including a compatibility condition on design matrix $(\mb X, \mb Z)$ \citep{BuhlmannGeer2009})  we will have that $\hat{\sigma} \inprob \sigma$ and $\|\hat{\beta}^Y - \beta^Y\|_1 \lesssim s_Y \sqrt{\log(p)/n}$ with high probability, where $s_Y:=|\{j:\beta^Y_j \neq 0\}|$.
		\begin{equation} \label{eq:Holder}
			\begin{split}
							\frac{|(\beta^Y-\check{\beta}^Y)^T\mb Z^T \mb R|}{\|\mb R\|_2} & \lesssim \lambda_X s_Y \sqrt{\log(p)} \\
							& \lesssim s_Y\log(p) / \sqrt{n},
			\end{split}
		\end{equation}
		where we have used H\"older's inequality and \eqref{eq:KKT_X}.
		Thus, in an asymptotic regime where $s_Y\log(p) / \sqrt{n} \to 0$, Slutsky's Lemma gives us that $T_{\DB} \indist \mathcal{N}(0, 1)$.
		
		Note that essentially no assumptions regarding a regression model for $\mb X$ on $\mb Z$ are required here; the only purpose of the square-root Lasso regression producing $\hat{\beta}^X$ is to construct the vector of residuals $\mb R$. This latter quantity may be regarded as a version of predictor $\mb X$ modified to be almost orthogonal to the remaining covariates $\mb Z$ \eqref{eq:KKT_X} such that when normalised, the dot product with the bias term $\mb Z (\beta^Y - \check{\beta}^Y)$ is well-controlled \eqref{eq:Holder}. Although this orthogonality comes free as a by-product of the square-root Lasso, we have however tacitly assumed $\mb R \neq \mb 0$ to arrive at \eqref{eq:KKT_X}. If $\mb R=\mb 0$ (which we have yet to observe in practice) we can simply agree to accept the null of conditional independence, so this poses no problem for type I error control.
		We note that the same sort of orthogonality argument 
		may not go through for a regular Lasso estimator with tuning parameter chosen by cross-validation, for example, as control of the LHS of \eqref{eq:KKT_X} with no assumptions on the model would be very challenging. However, empirically, we have observed that the cross-validated Lasso performs similarly to the square-root Lasso here.
		
		Now consider the case where the $X$-model is a sparse linear model. Whilst we will have control of $\|\beta^X - \hat{\beta}^X\|_1$, the equivalent of \eqref{eq:KKT_X} with  residuals $\mb R$ replaced by $\mb Y - \mb Z \check{\beta}^Y$ will not hold in general. The issue is that the latter quantity is not equal to the residuals from the $Y$-regression unless $\hat{\theta}=0$. Thus the debiased Lasso is not quite DEF in that it can be sensitive to misspecification of the $Y$-model.
		
		There are several options for how to restore the DEF property in this setting, but one that is particularly simple involves enforcing that $\hat{\theta}=0$, that is setting $\hat{\beta}^Y$ to be coefficients from a regression of $\mb Y$ on $\mb Z$ rather than the augmented design $(\mb X, \mb Z)$:
		\begin{equation} \label{eq:beta_hat_DEF}
			\hat{\beta}^Y := \argmin_{b \in \R^p} \{\|\mb Y - \mb Z b\|_2/\sqrt{n} + \lambda_Y \|b\|_1\};
		\end{equation}
		note this differs from the definition in \eqref{eq:DB_lasso}.
		The resulting test statistic takes the form of a regularised partial correlation:
		\begin{equation} \label{eq:T_DEF}
			T_{\DEF} := T_{\DEF}(\mb Y, \mb X) := \sqrt{n} \frac{(\mb Y - \mb Z \hat{\beta}^Y)^T(\mb X - \mb Z \hat{\beta}^X)}{\|\mb Y - \mb Z \hat{\beta}^Y\|_2 \|\mb X - \mb Z \hat{\beta}^X\|_2};
		\end{equation}
		note the inclusion of the notation $T_{\DEF}(\mb Y, \mb X)$ making the dependence of the test statistic on $\mb Y$ and $\mb X$ is included here for use later in Section~\ref{sec:high_conf}.
		In the unlikely case that the denominator defining $T_{\DEF}$ above is zero, so one of the square-root Lasso solutions is degenerate, we will set $T_{\DEF}=0$; we have never observed this degeneracy to occur in any of the numerical experiments conducted.
		The test statistic \eqref{eq:T_DEF} above was first studied in \citet{ren2015asymptotic} in the context of Gaussian graphical model estimation where asymptotic normality was shown when both the $X$-model and $Y$-model are sparse. The work of \citet{shah2018goodness} extended this result to show that the same conclusion holds when only the $Y$-model holds, and hence by symmetry of the test statistic, that it has the DEF property.
		Below we state a variant of the latter result that allows for non-Gaussian errors.
		
		In the case that (only) the $Y$-model holds, we will need to assume in addition to (Y1) and (Y2) with $\mb R = \mb X - \mb Z \hat{\beta}^X$, the following conditions.
		\begin{itemize}
			\item[(Y3)] Defining $S_Y := \{j : \beta^Y_j \neq 0\}$ and $s_Y := |S_Y|$, we have $s_Y\log(p) / \sqrt{n} \to 0$.
			\item [(Y4)] $\|\hat{\beta}^Y - \beta^Y\|_1 = O_{\pr}(s_Y\sqrt{\log(p)/n})$.
			%There exists constant $C>0$ such that $\pr(\|\hat{\mbb\beta}^Y - \mbb\beta^Y\|_1 > Cs_Y\sqrt{\log(p)/n} ) \to 0$.
			\item[(Y5)] $\|\mb Y - \mb Z \hat{\beta}^Y\|_2^2/n \inprob \sigma^2$.
			%\item[(Y6)] Defining $\mb R := \mb X - \mb Z \hat{\beta}^X$ we have that \eqref{eq:R} holds.
			%\[
			%\frac{1}{\sqrt{n}} \E \Bigg( \frac{\frac{1}{n}\sum_{i=1}^n |R_i|^3}{\Big(\frac{1}{n}\sum_{i=1}^n R_i^2\Big)^{3/2}} \ind_{\{\mb R_n \neq \mb 0\}}\Bigg) \to 0.
			%\]
		\end{itemize}
		Note that, as in the low-dimensional case, the only assumption placed on the conditional distribution of
		$\mb X$ given $\mb Z$ is (Y2), with $\mb R = \mb X - \mb Z \hat{\beta}^X$. This would be satisfied if we had a sparse
		linear $X$-model, but such an assumption is very far from necessary in
		order for (Y2) to hold. Furthermore, as shown in \citet{shah2018goodness},
		this is not necessary when the errors $\mbb\varepsilon$ for the $Y$-model
		are Gaussian. We also introduce, in addition to (X1) and (X2) with $\mb R = \mb Y - \mb Z \hat{\beta}^Y$, the following assumptions that are relevant when the $X$-model holds.
		\begin{itemize}
			\item[(X$j$)] As (Y$j$) above, but with $X$ and $\mb X$
			interchanged with $Y$ and $\mb Y$ everywhere, for $j \in \{3, 4, 5\}$.
		\end{itemize}
		%We will  let (X3)--(X6) be as above, 
		%\Peter{Perhaps highlight (X3)--(X6); see my comment for (X1)--(X2).}
		We have the following result.
		\begin{thm} \label{thm:DEF_lin}
			Let $\lambda_X=\lambda_Y = A\sqrt{2 \log(p)/n}$ for some $A > 1$.
			Assume that either (Y1)--(Y5) or (X1)--(X5) hold. Then under the null hypothesis that $\mb X \ci \mb Y \,|\, \mb Z$, test statistic $T_{\DEF}$ defined according to \eqref{eq:T_DEF} satisfies $T_{\DEF} \indist \mathcal{N}(0, 1)$.
		\end{thm}
		Similarly to the case with the debiased Lasso, under an alternative where $\mb Y = \mb X \theta + \mb Z \beta^Y + \mbb\varepsilon$, if a sparse linear $X$-model also holds, $T_{\DEF}$ \eqref{eq:T_DEF} has power tending to 1 when $\sqrt{n} \theta \to \infty$. We refer the reader to \citet{ren2015asymptotic} and \citet{shah2018goodness} for further details.
		
		The DEF version of the debiased Lasso bears some similarities to the CorrT test developed and studied in \citet{zhu2018significance}. However whereas the latter relies on estimating $\beta^Y$ and $\beta^X$ via a family of linear programs, the DEF statistic presented here can be calculated using standard software for computing Lasso solutions such as \texttt{glmnet} \citep{friedman2010regularization}. 
		We note further that whereas Theorem~\ref{thm:DEF_lin} only requires the weak condition that no residual from the regression relating to the misspecified is too extreme (and no condition on the residuals when the errors in the true model are Gaussian), the corresponding result (Theorem 2) in \citet{zhu2018significance} requires the misspecified model to nevertheless be a linear model with the coefficient vector having bounded $\ell_2$-norm. Furthermore the sparsity condition $s = o(\sqrt{n} /  (\log p)^{5/2})$ is assumed, where $s$ is the sparsity of the coefficient vector in the well specified model, compared to our requirement of $s = o(\sqrt{n} / \log p)$. On the other hand, the CorrT test accommodates heteroscedastic errors whereas one would need to modify our statistic to
		\[
		\sqrt{n}\frac{\frac{1}{n}(\mb R^Y)^T \mb R^X}{\frac{1}{n}\sum_{i=1}^n (\mb R^Y_i)^2 (\mb R^X_i)^2 - \left(\frac{1}{n}(\mb R^Y)^T \mb R^X\right)^2 },
		\]
		where
		\[
		\mb R^Y := \mb Y - \mb Z \hat{\beta}^Y \qquad \text{and} \qquad \mb R^X := \mb X - \mb Z \hat{\beta}^X
		\]
		in order to achieve this; see \citet{shah2018hardness} which uses the denominator above more generally in nonparametric models.
		\subsection{Confidence intervals via inverting tests} \label{sec:high_conf}
		%\Rajen{I think we discussed we could somehow try to make this more prominent?}
		Thus far we have only discussed testing, but using the DEF statistic \eqref{eq:T_DEF}, it is straightforward to obtain confidence intervals for a parameter $\theta$ in the partially linear model
		\begin{equation} \label{eq:par_lin}
			\mb Y = \mb X \theta + f(\mb Z, \mbb\varepsilon)
		\end{equation}
		where $\mbb\varepsilon \ci \mb X \,|\, \mb Z$ and $f : \R^{n \times p} \times \R^n \to \R^n$ under the following conditions:
		either $f(\mb Z, \mbb\varepsilon) = \mb Z \beta^Y + \mbb\varepsilon$, or a sparse linear $X$-model holds.
		Our approach for constructing a confidence region for $\theta$ utilises the well-known duality between confidence intervals and hypothesis tests; specifically we invert the DEF test, noting that under \eqref{eq:par_lin}, we have $\mb Y - \mb X\theta \ci \mb X \,|\, \mb Z$. We first compute test statistic
		\begin{equation} \label{eq:confint}
			T_{\DEF,t} := T_{\DEF}(\mb Y - \mb X t, \mb X),
		\end{equation}
		that is we subtract $t$ times $\mb X$ from $\mb Y$ and compute the usual DEF test statistic. Then we form a $1-\alpha$ confidence region $R_\alpha$ via
		\[
		R_\alpha := \{t \in \R : |T_{\DEF,t}| \geq z_{\alpha}\} 
		\]
		where $z_\alpha$ is the upper $\alpha/2$ quantile of a standard normal distribution. As a consequence of Theorem~\ref{thm:DEF_lin} This confidence region has the following asymptotic validity.
		\begin{cor} \label{cor:conf}
			Suppose the partially linear model \eqref{eq:par_lin} holds with $\mbb\varepsilon \ci \mb X \,|\, \mb Z$ and let $\lambda_X=\lambda_Y = A\sqrt{2 \log(p)/n}$ for some $A > 1$. Suppose the assumptions of Theorem~\ref{thm:DEF_lin} hold with $\mb Y$ replaced by $\mb Y - \mb X\theta$, i.e., in particular either $f(\mb Z, \mbb\varepsilon) = \mb Z \beta^Y + \mbb\varepsilon$, or a sparse linear $X$-model holds. Then for any $\alpha \in (0, 1)$,
			\[
			\pr(\theta \in R_\alpha) = \pr(|T_{\DEF,\theta}|\geq z_{\alpha}) \to 1-\alpha.
			\]
		\end{cor}
		%We will have asymptotic validity,
		%\[
		%\pr(\theta \in R_\alpha) = \pr(|T_{\DEF,\theta}|\geq z_{\alpha}) \to 1-\alpha,
		%\]
		%if the assumptions of Theorem~\ref{thm:DEF_lin} hold with $\mb Y$ replaced by $\mb Y - \mb X\theta$.
		Interestingly, in the case where the $X$-model holds, $f$ can be a fairly exotic function such that different components of $f(\mb Z, \mbb\varepsilon) \in \R^n$ are dependent, provided (X6) holds. Figure~\ref{fig:confint} illustrates our construction.
		
		Rather than directly seeking for an estimate of $\theta$, by inverting hypothesis tests, we do not rely on being able to distinguish the contribution of $\mb X$ from among the remaining covariates $\mb Z$. Thus for example having $\mb X$ very highly correlated with $\mb Z$ would not interfere with coverage properties of the intervals.
		%In fact this simple scheme for producing confidence regions for $\theta$ in the partially linear model \eqref{eq:par_lin} can be used in conjunction with any conditional independence test
		
		Of course, computing $T_{\DEF,t}$ for all $t \in \R$ is not feasible. However, whilst $R_\alpha$ is not guaranteed to be an interval in general, it appears to be the case in practice and we have yet to find a counterexample. This observation allows us to find the end points of the interval via a bisection search. We use coordinate descent to solve the square-root Lasso programmes involved in computing the test statistics $T_{\DEF,t}$, and warm start this iterative optimisation procedure at the closest point computed in the search. Whilst this construction is computationally more intensive than the standard approach with the debiased Lasso, it is still feasible in large-scale settings.  For the example shown in Figure~\ref{fig:confint}, the computation of the $500$ confidence intervals taking each columns of $\mb Z$ as the variable of interested (i.e.\ treating it as $\mb X$) took under 6 seconds on a standard laptop; this time could be further reduced by performing computations in parallel.

		%Further details on how to compute the confidence intervals are provided in Section~\ref{sec:conf_comp} in the Appendix.
		\begin{figure}
			\centering
			\makebox{\includegraphics[width=0.45\textwidth]{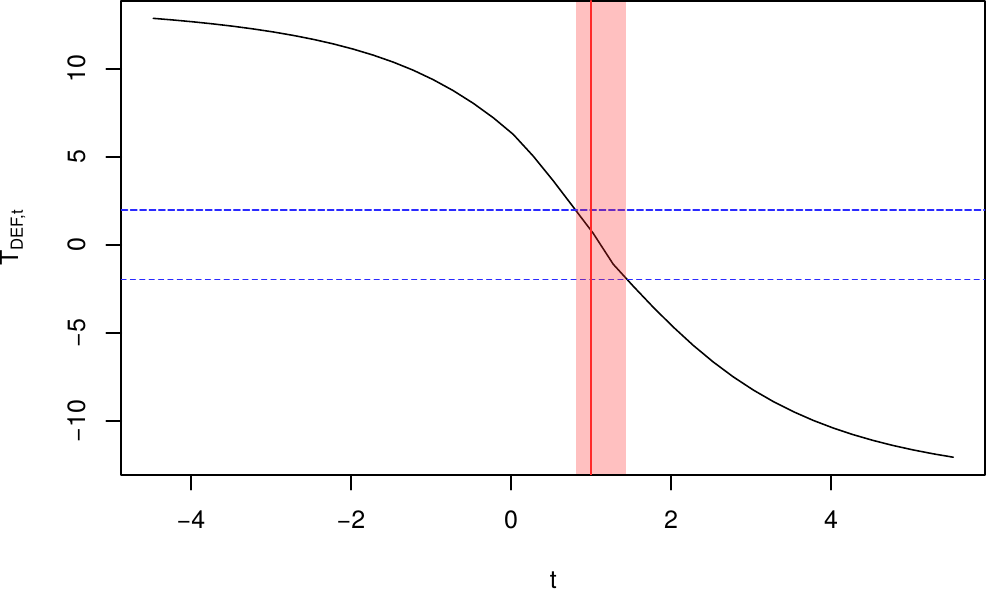}}
			\caption{Illustration of confidence interval construction. We generated $(\mb X, \mb Z) \in \R^{n \times p}$ with independent rows distributed as $\mathcal{N}_p(0, \Sigma)$ with $\Sigma_{jk} = 0.9^{|j-k|}$ where $(n, p) = (200, 500)$. A response $\mb Y$ was generated through $Y_i = X_i -0.5 Z_{i1} + 0.7 Z_{i2} + \varepsilon_i$ where $\mbb\varepsilon \sim \mathcal{N}_n(0, I)$. The plot shows $T_{\DEF,t}$ \eqref{eq:confint} as a function of $t$ (black curve). Horizontal dotted blue lines lie at $\pm  z_{0.05}$ and the shaded red region enclosing the intersection points with the curve $(t, T_{\DEF,t})$ depicts the 95\% confidence interval; here this contains the true parameter $\theta=1$.\label{fig:confint}}
		\end{figure}
		
		In Section~\ref{sec:pred_int} of the appendix we show how a similar technique to that described above can be used to construct confidence intervals for $w^T\beta^Y$ for some $w \in \R^p$ that is potentially dense, when the $Y$-model is a sparse linear model. This is perhaps most useful when $w$ is an additional covariate vector for a new observation whose corresponding response has not been observed; we can thus provide a confidence interval for the mean response conditional on the observed vector of covariates.
		
		\subsection{Generalised linear models} \label{sec:high_gen}
		We have seen in Section~\ref{sec:high_lin} how one can modify the debiased Lasso to construct a test statistic that has similar sorts of DEF properties to that enjoyed by the standard $t$-statistic in the low-dimensional setting. In Section~\ref{sec:low_gen} we saw how standard inference for generalised linear models has a DEF property, albeit with a slight modification needed to account for the different variances of the test statistics when the $Y$-model is misspecified. It is natural to ask whether   inferential procedures for high-dimensional generalised linear models can be adapted to be DEF, but one could equally ask the broader question of whether we can specify sparse generalised linear $X$ and $Y$-models (possibly different for each), and obtain valid inference if at most one of these is misspecified: this is the question we attempt to address here. As a first step in this direction, we consider heteroscedastic linear models, and then move on to treat generalised linear models in Section~\ref{sec:genmod}.
		
		\subsubsection{Heteroscedastic linear models} \label{sec:hetero}
		Consider the model $Y_i = Z_i^T \beta^Y +\zeta_i$ where $\E(\zeta_i\,|\, \mb Z)=0$, $\Var(\zeta_i \,|\, \mb Z) = \sigma_Y^2 / (D^Y_{ii})^2$ and the $\zeta_i$ are independent conditional on $\mb Z$; and a similar $X$-model. Equivalently, we may write
		\begin{align}
			\mb D^Y \mb Y &= \mb D^Y \mb Z \Lambda^Y \beta^Y +  \mbb\varepsilon^Y \label{eq:hetero_Y},\\
			\mb D^X \mb X &= \mb D^X \mb Z \Lambda^X \beta^X + \mbb\varepsilon^X \label{eq:hetero_X}
		\end{align}
		for the $Y$ and $X$-models respectively, where $\Var(\varepsilon^Y_i) =
		\sigma_Y^2$, $\Var(\varepsilon^X_i) = \sigma_X^2$ and the diagonal matrices
		$\Lambda^Y, \Lambda^X \in \R^{p\times p}$ are such that the empirical
		variances of the columns of the  resulting design matrices $\mb D^Y \mb Z
		\Lambda^Y$ and $\mb D^X \mb Z \Lambda^X$ are $1$. Note we have redefined
		$\beta^Y$ and $\beta^X$ by scaling them by $\Lambda^Y$ and $\Lambda^X$
		respectively. We will treat the diagonal matrices $\mb D^Y$ and $\mb D^X$
		as known, though one of \eqref{eq:hetero_Y} and \eqref{eq:hetero_X} may be
		misspecified, in which case the corresponding matrix will be
		meaningless.
		%\Peter{(but the equation does not hold anyway).} 
		In this context, it seems natural to seek an analogue of the test statistic $T_{\DEF}$ based on the weighted square-root Lasso regressions
		\begin{align*}
			\hat{\beta}^Y &= \argmin_{ b \in \R^{p}} \{\|\mb D^Y (\mb Y -\mb Z \Lambda^Y b)\|_2/\sqrt{n} + \lambda \| b\|_1 \}, \\
			\hat{\beta}^X &= \argmin_{ b \in \R^{p}} \{\|\mb D^X(\mb X -\mb Z \Lambda^X b)\|_2/\sqrt{n} + \lambda\| b\|_1 \}.
		\end{align*}
		The KKT conditions of the above optimisations are however not ``compatible'' in the same way as allowed for arguments similar to \eqref{eq:Holder}; the issue is that the design matrices in \eqref{eq:hetero_X} and \eqref{eq:hetero_Y} are different so Theorem~\ref{thm:DEF_lin} does not directly apply. Thus we cannot conclude that the bias term is small unless, for example, both the $X$ and $Y$-models specified above hold. Instead, consider orthogonalising the residuals $\tilde{\mb Y}:=\mb Y - \mb Z \Lambda^Y \hat{\beta}^Y$ and $\tilde{\mb X}:=\mb X - \mb Z \Lambda^X \hat{\beta}^X$ from the regressions above using the following construction:
		\begin{align}
			(\tilde{\beta}^Y, \tilde{\eta}^Y) &= \argmin_{( b, u) \in \R^{p} \times \R^{p}} \{ \|\mb D^Y(\tilde{\mb Y} - \mb Z \Lambda^Y b) \notag \\
			& \;\;\;\; - \mb D^X \mb Z  \Lambda^X u)\|_2/\sqrt{n} + \lambda(\| b\|_1 + \|u\|_1)\} \label{eq:orthog_Y}\\
			(\tilde{\beta}^X, \tilde{\eta}^X) &= \argmin_{( b,  u) \in \R^{p} \times \R^{p}} \{ \|\mb D^X(\tilde{\mb X} - \mb Z \Lambda^X b) \notag \\
			&\;\;\;\;- \mb D^Y \mb Z  \Lambda^Y u)\|_2/\sqrt{n} + \lambda(\| b\|_1 + \| u\|_1)\}. \label{eq:orthog_X}
		\end{align}
		Here we have augmented the designs with the terms $\mb D^X \mb Z$ and $\mb D^Y \mb Z$. The only purpose of these terms and the corresponding estimates $\tilde{\eta}^Y$ and $\tilde{\eta}^X$ is to ensure that the residuals from the regressions above satisfy the required near-orthogonality properties for controlling the bias term.

		Consider now the case that the $Y$-model \eqref{eq:hetero_Y} is well-specified.
		%Under reasonable conditions, the extra term $\mb D^X \mb Z$ should not hamper estimation of $\beta^Y$ and we should obtain that $\|\beta^Y - \hat{\beta}^Y\|_1$ and $\|\hat{\eta}^Y\|_1$ are small with high probability.
		Let $\mb R^X := \mb D^X (\tilde{\mb X} -\mb Z \Lambda^X\tilde{\beta}^X) - \mb D^Y \mb Z \Lambda^Y \tilde{\eta}^X$.
		The KKT conditions for \eqref{eq:orthog_X} yield in particular that
		\begin{equation} \label{eq:KKT_weighted}
			\begin{split}
				\frac{1}{\sqrt{n}}\frac{\|\Lambda^X \mb Z^T \mb D^X \mb R^X\|_\infty}{\|\mb R^X\|_2} &\leq \lambda \\
				\frac{1}{\sqrt{n}}\frac{\|\Lambda^Y \mb Z^T \mb D^Y \mb R^X\|_\infty}{\|\mb R^X\|_2} &\leq \lambda \, ;
			\end{split}
		\end{equation}	
%		\begin{align} \label{eq:KKT_weighted}
%			\frac{1}{\sqrt{n}}\frac{\|\Lambda^X \mb Z^T \mb D^X \mb R^X\|_\infty}{\|\mb R^X\|_2} \leq \lambda \qquad \text{and} \qquad
%			\frac{1}{\sqrt{n}}\frac{\|\Lambda^Y \mb Z^T \mb D^Y \mb R^X\|_\infty}{\|\mb R^X\|_2} \leq \lambda \, ;
%		\end{align}
		note the second inequality is due to the additional $\mb D^Y \mb Z\Lambda^Y$ term included in \eqref{eq:orthog_X}. Let us also define $\mb R^Y$ to be the equivalent of $\mb R^X$, but with $X$ and $\mb X$ interchanged everywhere with $Y$ and $\mb Y$ respectively. With these we may define a weighted version of the test statistic $T_{\DEF}$ which is simply a scaled correlation between the weighted residuals $\mb R^X$ and $\mb R^Y$:
		\begin{equation*}
			T_{\WDEF} := \sqrt{n} \frac{(\mb R^X)^T \mb R^Y }{\|\mb R^X\|_2 \|\mb R^Y\|_2}.
		\end{equation*}
		Similar to the homoscedastic case, we set $T_{\WDEF}=0$ if the denominator above is zero.
		We now explain why we will typically have $T_{\WDEF} \indist \mathcal{N}(0, 1)$ if the $Y$-regression holds, and hence also by symmetry, if the $X$-regression holds. Let us write $\hat{\sigma} := \|\mb R^Y\|_2/\sqrt{n}$. Now
		\begin{equation} \label{eq:weighted_bias}
			\mb R^Y = \mbb\varepsilon^Y +\mb D^Y \{\mb Z \Lambda^Y (\beta^Y - \hat{\beta}^Y - \tilde{\beta}^Y)\} - \mb D^X\mb Z\Lambda^X\tilde{\eta}^Y.
		\end{equation}
		Thus we have
		\begin{align*}
			\hat{\sigma} T_{\WDEF} &=  (\mbb\varepsilon^Y)^T \frac{\mb R^X}{\|\mb R^X\|_2}   \\
			&\;\;\;\;+ (\beta^Y - \hat{\beta}^Y-\tilde{\beta}^Y)^T\Lambda^Y \mb Z^T \mb D^Y \frac{\mb R^X}{\|\mb R^X\|_2} \\
			&\;\;\;\;- (\tilde{\eta}^Y)^T \Lambda^X \mb Z^T \mb D^X \frac{\mb R^X}{\|\mb R^X\|_2}\\
			&=: \text{(i)}+\text{(ii)}+\text{(iii)}.
		\end{align*}
		Under weak conditions, the first term (i) will converge in distribution to a normal distribution. The two sets of near-orthogonality conditions \eqref{eq:KKT_weighted} in conjunction with H\"older's inequality give that the two bias terms above satisfy
		\begin{align*}
		|\text{(ii)}| &\leq \sqrt{n}\lambda(\|\beta^Y - \hat{\beta}^Y\|_1 + \|\tilde{\beta}\|_1), \\
		|\text{(iii)}| &\leq \sqrt{n} \lambda \|\hat{\eta}^Y\|_1
		\end{align*}
		respectively. As explained in Section~\ref{sec:high_lin}, we can expect that under reasonable conditions we have $\|\beta^Y - \hat{\beta}^Y\|_1 \lesssim s_Y \sqrt{\log(p)/n}$ with high probability. The additional terms $\|\tilde{\beta}^Y\|_1$ and $\|\tilde{\eta}^Y\|_1$ may be controlled similarly to $\|\beta^Y - \hat{\beta}^Y\|_1$; see Theorem~\ref{thm:l_1_bd} below.
		
		%Note that the restriction to Gaussian error above is mainly for simplicity,
		%and similar results would hold for settings with sub-Gaussian errors, for
		%example.
		Throughout the discussion above, we have assumed that the $Y$-model holds. If instead the $X$-model is correct, the symmetry of the test statistic allows that analogous results may be established in the same manner, justifying that $T_{\WDEF}$ has a standard normal distribution under the null-hypothesis if either model is well-specified. This is formalised in the result below, which assumes some additional moment conditions for the entries in $\mb Z$, and a condition on the growth rate of $p$ compared to $n$.
		\begin{thm} \label{thm:l_1_bd}
			Suppose there exist constants $M, \delta > 0$ such that
			\[
			\pr\left(\frac{1}{n}\sum_{i=1}^n (|D^Y_{ii}Z_{ij}\Lambda^Y_{ii}|^{2+\delta} + |D^X_{ii}Z_{ij}\Lambda^X_{ii}|^{2+\delta}) \leq M\right)\to 1,
			\]
			and $p \leq n^{c \delta}$ for some $c \in (0,1)$ and all $n$ sufficiently large.
			Suppose that (Y1) holds with the heteroscedastic $Y$-model \eqref{eq:hetero_Y} in place of the linear model and $\delta$ as above, (Y2) holds with $\mb R = \mb R^X$,  and (Y3)--(Y5) hold.
			%, or the equivalent of all these conditions hold with $X$ and $Y$ interchanged.
			Suppose $\lambda = A\sqrt{2\log(p)/n}$ for some $A>1$. Then there exists a constant $C>0$ such that
			\begin{equation} \label{eq:l_1_control}
				\pr\left(\|\tilde{\beta}^Y\|_1 + \|\tilde{\eta}^Y\|_1 \leq C \|\beta^Y - \hat{\beta}^Y\|_1 \right) \to 1,
			\end{equation}
			and moreover, under the null hypothesis that $\mb X \ci \mb Y \,|\, \mb Z$, we have $T_{\WDEF} \indist \mathcal{N}(0, 1)$.
		\end{thm}
		By symmetry, an analogous version of the result holds with every instance of $Y$ and $\mb Y$ interchanged with $X$ and $\mb X$ respectively.
		\subsubsection{Generalised linear models} \label{sec:genmod}
		With the methodology for heteroscedastic linear models introduced above, we can now set out a DEF test statistic for the case where we wish to specify the $X$ and $Y$-models as generalised linear models.
		The first step is to run penalised generalised linear regressions of each of $\mb Y$ and $\mb X$ on $\mb Z$ to obtain coefficient estimates $\hat{\beta}^Y, \hat{\beta}^X \in \R^p$.
		Let $\mu_X$ and $\mu_Y$ be the respective mean functions (i.e.\ inverse link functions) so that if the $Y$-model is well-specified and $\mb Y \ci \mb X \,|\, \mb Z$, we have $\E(Y_i \,|\, Z_i) = \mu_Y(Z_i^T\beta^Y)$ where $\beta^Y \in \R^p$. Further define variance functions $V_{Y,i}$ for the $Y$-model; when the $Y$-model holds we will have $V_{Y,i}(\mu_Y(Z_i^T\beta^Y)) = \Var(Y_i\,|\,Z_i)$. We will assume for simplicity that the $V_{Y,i}$ are known and do not vary over the observations, so we may write $V_Y=V_{Y,i}$.
		Define the variance function $V_X$ for the $X$-model analogously.
		
		To compute a DEF test statistic for generalised linear models, we take the following steps.
		\begin{enumerate}
			\item Define the adjusted response $\tilde{\mb Y} \in \R^n$ by
			\[
			\tilde{Y}_i := \frac{Y_i - \mu_Y(Z_i^T\hat{\beta}^Y)}{\mu_Y'(Z_i^T\hat{\beta}^Y)}
			%+ Z_i^T\hat{\beta}^Y
			\]
			and define $\tilde{\mb X}$ analogously.
			\item Define diagonal matrix $\hat{\mb D}^Y \in \R^{n \times n}$ by $\hat{D}^Y_{ii} = \mu_Y'(Z_i^T\hat{\beta}^Y) \{V_Y(\mu_Y(Z_i^T\hat{\beta}^Y))\}^{-1/2}$, and define $\hat{\mb D}^X$ analogously.
			\item Compute test statistic $T_{\GENDEF}$ by forming $T_{\WDEF}$ but replacing $\mb X$ and $\mb Y$ with their adjusted versions $\tilde{\mb X}$ and $\tilde{\mb Y}$, and using the diagonal matrices $\hat{\mb D}^X$ and $\hat{\mb D}^Y$ defined above. 
		\end{enumerate}
		
		We now explain why we can expect that $T_{\GENDEF} \indist \mathcal{N}(0, 1)$ when $\mb X \ci \mb Y \,|\, \mb Z$ and either the $Y$-model or $X$-model is well-specified. Suppose that the $Y$-model holds. Then a first order Taylor expansion yields
		\begin{align*}
			Y_i - \mu_Y(Z_i^T\hat{\beta}^Y) &= \mu_Y(Z_i^T\beta^Y)-\mu_Y(Z_i^T\hat{\beta}^Y) + \zeta_i \\
			&\approx Z_i^T(\beta^Y - \hat{\beta}^Y)\mu_Y'(Z_i^T\hat{\beta}^Y) + \zeta_i
		\end{align*}
		where $\E(\zeta_i\,|\,Z_i)=0$ and $\Var(\zeta_i\,|\,Z_i) =V_Y(\mu_Y(Z_i^T\beta^Y))$. Thus $\tilde{Y}_i \approx Z_i^T(\beta^Y - \hat{\beta}^Y) + \zeta_i / \mu_Y'(Z_i^T\hat{\beta}^Y)$ and hence
		\[
		\hat{\mb D}^Y \tilde{\mb Y} \approx \hat{\mb D}^Y \mb Z (\beta^Y - \hat{\beta}^Y) + \mbb\varepsilon,
		\]
		where $\E(\mbb\varepsilon\,|\, \mb Z) =\mb0$ and $\Var(\mbb\varepsilon\,|\, \mb Z)=\mb I$.
		
		Now the square-root Lasso regression involving $\tilde{\mb Y}$ used in step
		3 above should have little effect as $\tilde{\mb Y}$ is essentially noise
		(see Theorem~\ref{thm:l_1_bd}). The corresponding regression for $\tilde{\mb
			X}$ however will ensure the resulting residuals are almost orthogonal to
		the bias term $\hat{\mb D}^Y \mb Z (\beta^Y - \hat{\beta}^Y)$. Arguing
		similarly to \eqref{eq:weighted_bias}, we see that the overall bias should be well-controlled. The variance term $\mbb\varepsilon^T \mb R^X / \|\mb R^X\|_2$ should behave roughly like a weighted sum of independent zero-mean random variables $\varepsilon_i$. The fact that $\hat{\mb D}^Y$ is used in the construction of the residuals $\mb R^X$ however means they are not independent of $\mbb\varepsilon$, and one cannot directly apply a version of the central limit theorem to the term. Whilst some form of sample splitting could in principle help with this technical issue (see for example \citet{jankova2020goodness} where sample splitting is used in a similar context), as the dependence is weak, a normal approximation should work well in practice; indeed we show empirically in Section~\ref{sec:experiments} that this is the case.
		%We see that the $Y$-model may be transformed into an approximate linear regression model with design matrix $\hat{\mb D}^Y \mb Z$. The setting is therefore similar to that described in Section~\ref{sec:hetero}, and so we can expect that $T_{\GENDEF} \indist \mathcal{N}(0, 1)$.
		
		\subsubsection{Connections to the generalised covariance measure, the decorrelated score test and the debiased Lasso} \label{sec:score}
		An alternative to the approach for DEF inference in high-dimensional generalised linear models presented in the previous sections is based on the score test. Considering the setup of Section~\ref{sec:low_gen}, the key argument that results in the DEF property for maximum likelihood estimation in low-dimensional generalised linear models is that $\beta^\dagger$ defined as the maximiser of $\E \ell(Z^T\beta; Y)$ over $\beta \in \R^p$ satisfies
		\begin{equation} \label{eq:score_test}
			\E \{ X U(Z^T\beta^\dagger;Y)\} = \E\{(X - Z^T\beta^X) U(Z^T\beta^\dagger;Y)\}
		\end{equation}
		where $\beta^X := \argmin_{\beta \in \R^p} \E\{(X - Z^T\beta)^2\}$ is the best linear predictor of $X$ based on $Z$. It is straightforward to see that if $X \ci Y \,|\, Z$, the RHS is always zero whenever $Z^T\beta^X$ coincides with $\E (X \,|\, Z)$, and clearly the LHS (and hence also the RHS) is zero whenever the model \eqref{eq:log_lik} is well-specified.
		
		The RHS of \eqref{eq:score_test} may be used as the basis of a score-type test involving linearly regressing $\mb X$ onto $\mb Z$, and forming the empirical covariance of these residuals and  $\big(U(Z_i^T\check{\beta}^Y;Y_i)\big)_{i=1}^n$, where $\check{\beta}^Y$ is a maximum likelihood estimate of $\beta^Y$. Given that both regressions of $\mb X$ and $\mb Y$ on $\mb Z$ are performed to produce such a test statistic, it is more intuitively clear that this would have a DEF property. The $\mb X$ on $\mb Z$ regression is however redundant as the stationarity conditions of $\check{\beta}^Y$ dictate that $\big(U(Z_i^T\check{\beta}^Y;Y_i)\big)_{i=1}^n$ is orthogonal to the column space of $\mb Z$. Thus a regular score test would have the DEF property for a linear regression model of $X$ on $Z$.
		
		In high-dimensional settings  the estimate $\check{\beta}^Y$ will necessarily only yield approximate orthogonality to $\mb Z$, and so the regression of $\mb X$  on $\mb Z$ is crucial. In a setting where the regression for $Y$-model is a generalised linear model with canonical link, this leads to a test statistic of the form
		\begin{equation} \label{eq:GCM}
			\frac{1}{\hat{\tau}_D} \sum_{i=1}^n (X_i - Z_i^T\check{\beta}^X)^T\{Y_i - \mu_Y(Z_i^T\check{\beta}^Y)\},
		\end{equation}
		where $\hat{\tau}_D$ is a normalisation term that ensures an asymptotically unit variance under the null. This is the form of the generalised covariance measure (GCM) \citep{shah2018hardness}, the decorrelated score test \citep{ning2017general}, and, to a first order Taylor approximation, the debiased Lasso \citep{vandegeer2014}; however they differ primarily in their choice of estimates $\check{\beta}^X$ and $\check{\beta}^Y$. Both the GCM and the decorrelated score construct $\check{\beta}^Y$ through only regressing on $\mb Z$, similarly to our DEF approach, whereas the debiased Lasso involves a regression on $(\mb X, \mb Z)$. Like our approach, the $\mb X$ on $\mb Z$ regression in the GCM is performed without using $\mb Y$ and can be tailored to a specified $X$-model, whereas both the decorrelated score test and the debiased Lasso aim to construct $\check{\beta}^X$ so that the residuals $\mb X - \mb Z \check{\beta}^X$ are orthogonal to the bias in the residuals from $\mb Y$ regression, were the $Y$-model to be correct. Our DEF approach instead employs an orthogonalisation step using the square-root Lasso corresponding to each of $X$ and $Y$ after initial $\mb X$ and $\mb Y$ regressions have been performed. A further difference is that whereas \eqref{eq:GCM} involves an empirical covariance between raw residuals, our DEF approach uses Pearson residuals. This is so that the square-root Lasso orthogonalisation corresponding to the true model is performed on data with (approximately) homoscedastic errors, which permits \eqref{eq:l_1_control} to hold. We have however found that a version of the test with raw residuals performs very similarly in terms of power and type I error control.% though the analogue of Theorem~\ref{thm:l_1_bd} would require additional conditions to ensure that the square-root Lasso orthogon

		\section{Numerical experiments} \label{sec:experiments}
		% Real and Toeplitz. Lin, nonlin1, nonlin2. n/p two settings. Active vars 1:12. 12 plots in total. Here can plot only real and small n/p and omit nonlin1. Only two plots.
		In this section we explore the empirical properties of our proposed DEF methodology set out in Section~\ref{sec:high}.
		
		\subsection{Partially linear models} \label{sec:exp_lin}
		Here we investigate the empirical performance of our DEF confidence interval construction described in Section~\ref{sec:high_conf}, and compare it with the debiased Lasso.
		We consider partially linear regression models of the form
		\[
		Y_i = \theta X_i + f(Z_i, \varepsilon_i),
		\]
		where the goal is to provide a confidence interval for $\theta$. The nuisance function $f$, parameter $\theta$ and data $(Y_i, X_i, Z_i, \varepsilon_i) \in \R \times \R \times \R^p \times \R$ for $i=1,\ldots,n$ with $n=100$ are generated as follows. We use the publicly available gene expression data of Bacillus Subtilis \citep{Buehlmann2014}, which has $71$ observations and $4088$ predictors. We first select the $p+1=500$ predictors with the highest empirical variances, and then centre and scale these so the empirical variances are $1$. We then fit a Gaussian copula model to these predictors to give a $500$-dimensional multivariate distribution $P$ from which we can generate independent realisations of $(X_i, Z_i)$. This distribution is non-Gaussian and has some large pairwise correlations and thus is helpful for assessing how our methods may perform in challenging and realistic settings.
		
		To form $(X_i, Z_i)_{i=1}^n$ we first generate $(W_i)_{i=1}^n \iid P$ and then consider 12 settings taking each of the first $12$ components of $W_i$ as the variable $X_i$ of interest, and collecting the remaining components into $Z_i$. For each of the 12 settings, we generate a new $\theta \sim U[0, 2]$, and look at $3$ forms for the nuisance function $f$.
		\begin{enumerate}[(a)]
			\item \textbf{Linear.} We set
			\[
			f(Z_i, \varepsilon_i) = \sum_{j=1}^{11} Z_{ij}\beta_j + \varepsilon_i
			\]
			where the $(\beta_j)_{j=1}^{11}$ are generated independently and follow a $U[0,2]$. distribution.
			\item \textbf{Slightly nonlinear.} We set
			\[
			f(Z_i, \varepsilon_i) = \sum_{j=1}^{11} \tilde{Z}_{ij}\beta_j + \varepsilon_i
			\]
			with $(\beta_j)_{j=1}^{11}$ as in (a) and $\tilde{Z}_{ij} := 2e^{Z_{ij}}/(1+e^{Z_{ij}})-1$.
			\item \textbf{Highly nonlinear.} We first form
			\[
			\eta_i := \sum_{j=1}^{11} \tilde{Z}_{ij}\beta_j + \sum_{j=1}^{11}\sum_{k=1}^{11}
			\tilde{Z}_{ij}\tilde{Z}_{ik}\theta_{jk} + \varepsilon_i,
			\]
			where the $(\tilde{Z}_{ij})_{j=1}^{11}$ and $(\beta_j)_{j=1}^{11}$ are as above and $(\theta_{jk})_{j,k=1}^{11}$. We then set $f(Z_i, \varepsilon_i) = e^{\eta_i} / (1+e^{\eta_i})$.
		\end{enumerate} 
		In all cases the errors $(\varepsilon_i)_{i=1}^n$ are taken to be i.i.d.\ standard normal. In our implementation of the debiased Lasso and DEF confidence intervals, we use the square-root Lasso with parameters $\lambda_X$ and $\lambda_Y$ chosen according to the method of \citet{sun2013sparse}.
		Figures~\ref{fig:n1_real_lin}, \ref{fig:n1_real_nonlin1} and \ref{fig:n1_real_nonlin2} show the results. We see that the DEF 95\% confidence intervals have significantly better coverage compared to those based on the debiased Lasso. This is even true in the linear setting where one might have expected the performances to be similar, suggesting that the strategy of inverting hypothesis tests may also be useful when applied in conjunction with debiased Lasso-based tests. The improved coverage we observe is partly due to the DEF confidence intervals being wider, but they also seem to have slightly better centring around the true parameter values; in contrast the debiased Lasso confidence intervals display a substantial bias towards zero in several cases.
		
		Note that the nonlinear settings (b) and (c) do not quite satisfy the conditions for our theory (see Theorem~\ref{thm:DEF_lin}) as the non-Gaussianity of the $Z_i$ would mean that the $X$-models are unlikely to be sparse linear models. Nevertheless, the coverage is reasonable if not perfect in these more challenging settings.
		Results for analogous scenarios to those studied here but with $P$ replaced by a multivariate Gaussian with a Toeplitz covariance matrix $\Sigma$ where $\Sigma_{jk} = 0.9^{|j-k|}$ are shown in Section~\ref{sec:add_exp} of the appendix. In these settings, the $X$-model is a highly sparse linear model, and as a result the coverage properties of both methods are improved; however the debiased Lasso still undercovers whilst the DEF confidence intervals reach a coverage of closer to 95\%. We have observed a very similar pattern of results for other settings of $(n, p)$.
		\begin{figure*}
			\centering
			\includegraphics[width=\textwidth]{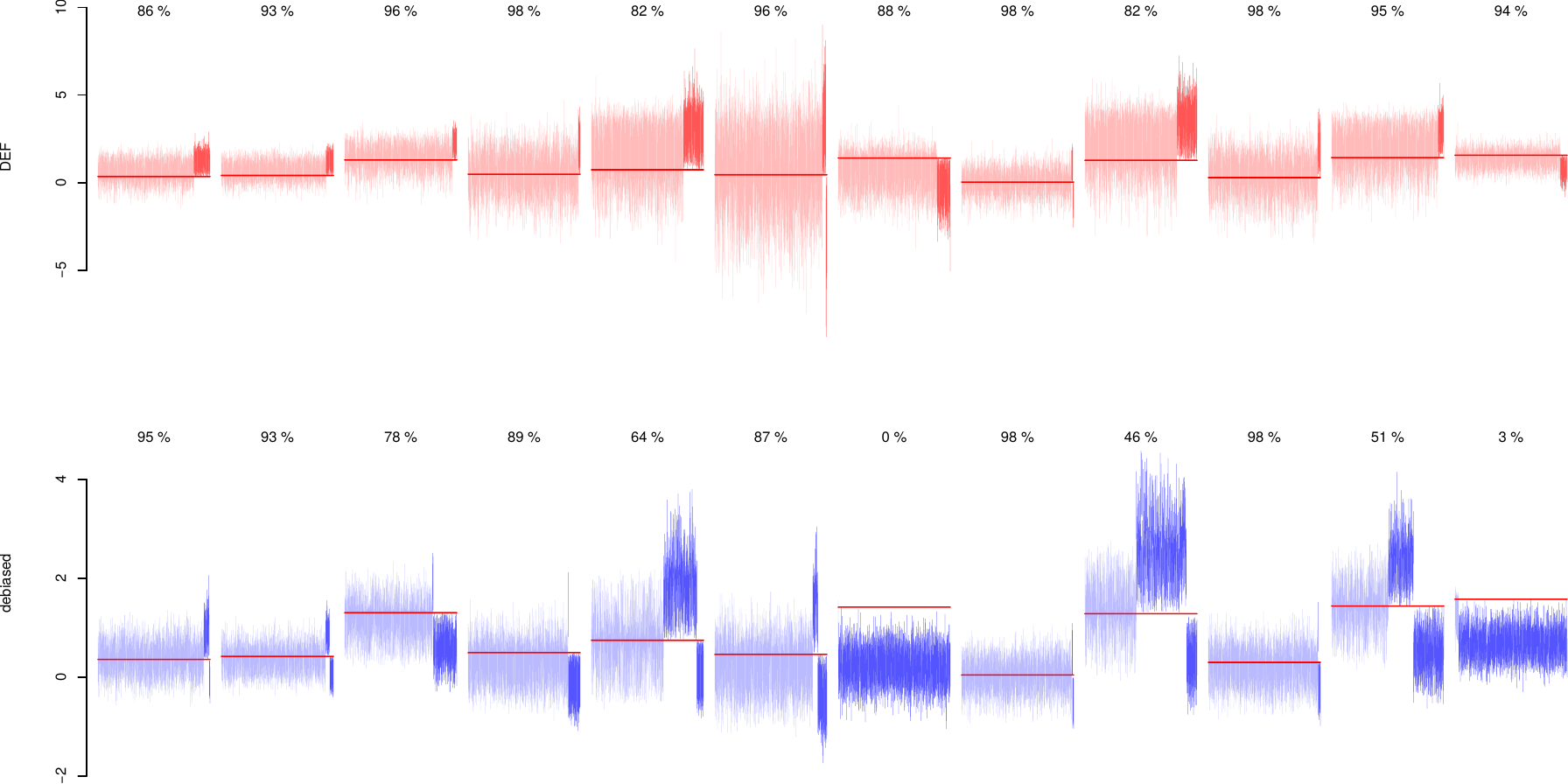}
			\caption{DEF (top row) and debiased Lasso (bottom row) 95\% confidence intervals from $500$ simulations of each of the 12 linear settings (a). The light red and blue vertical lines depict those confidence intervals that covered their target parameter $\theta$ shown the red horizontal lines. Darker vertical lines are confidence intervals that failed to cover their target and are grouped into those whose endpoints were too high, and too low. Coverage proportions are reported above each of the plots.\label{fig:n1_real_lin}}
		\end{figure*}
		\begin{figure*} 
			\centering
			\includegraphics[width=\textwidth]{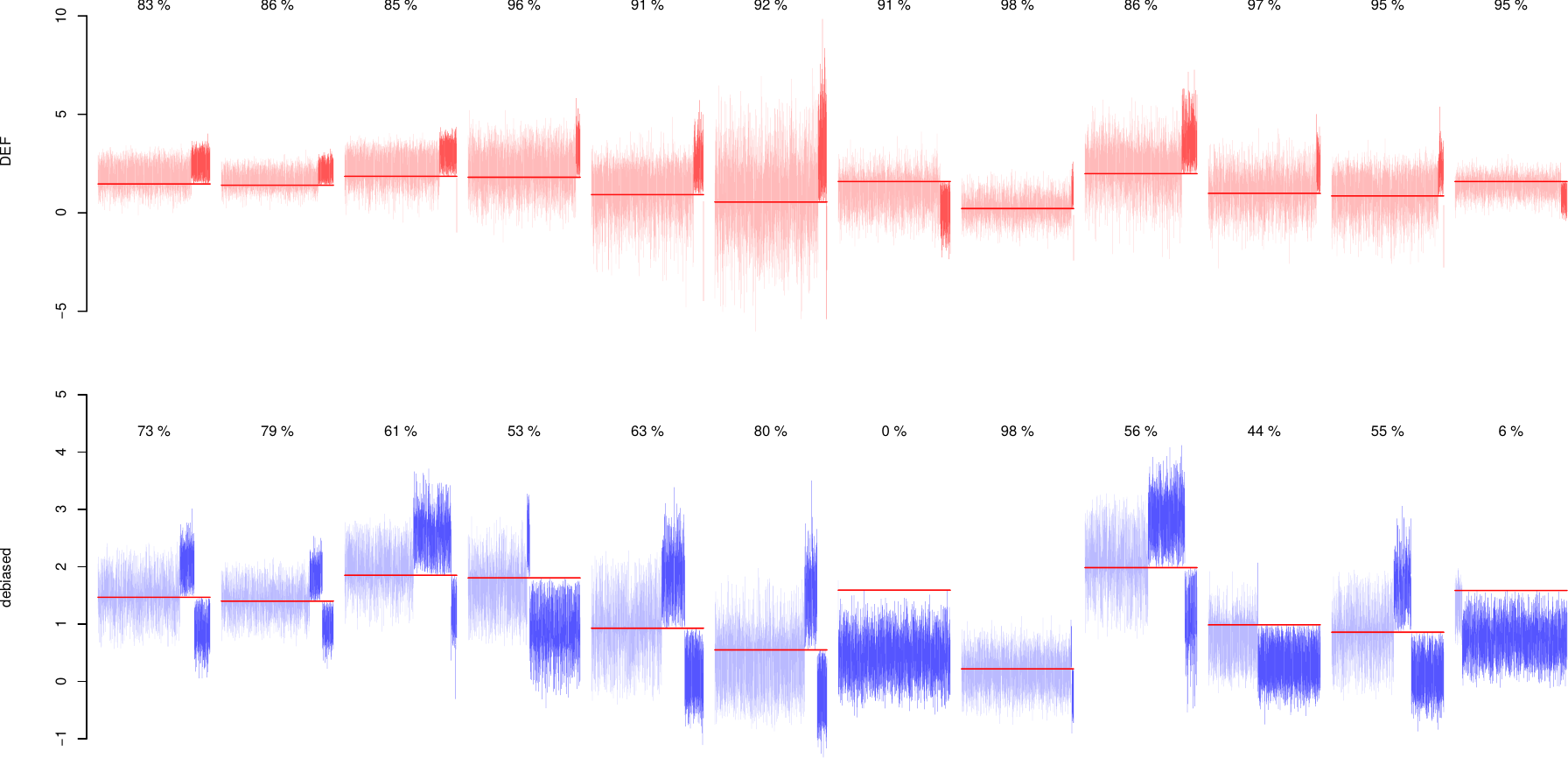}
			\caption{The slightly nonlinear setting (b); the interpretation is similar to that of Figure \ref{fig:n1_real_lin}.\label{fig:n1_real_nonlin1}}
		\end{figure*}
		\begin{figure*} 
			\centering
			\includegraphics[width=\textwidth]{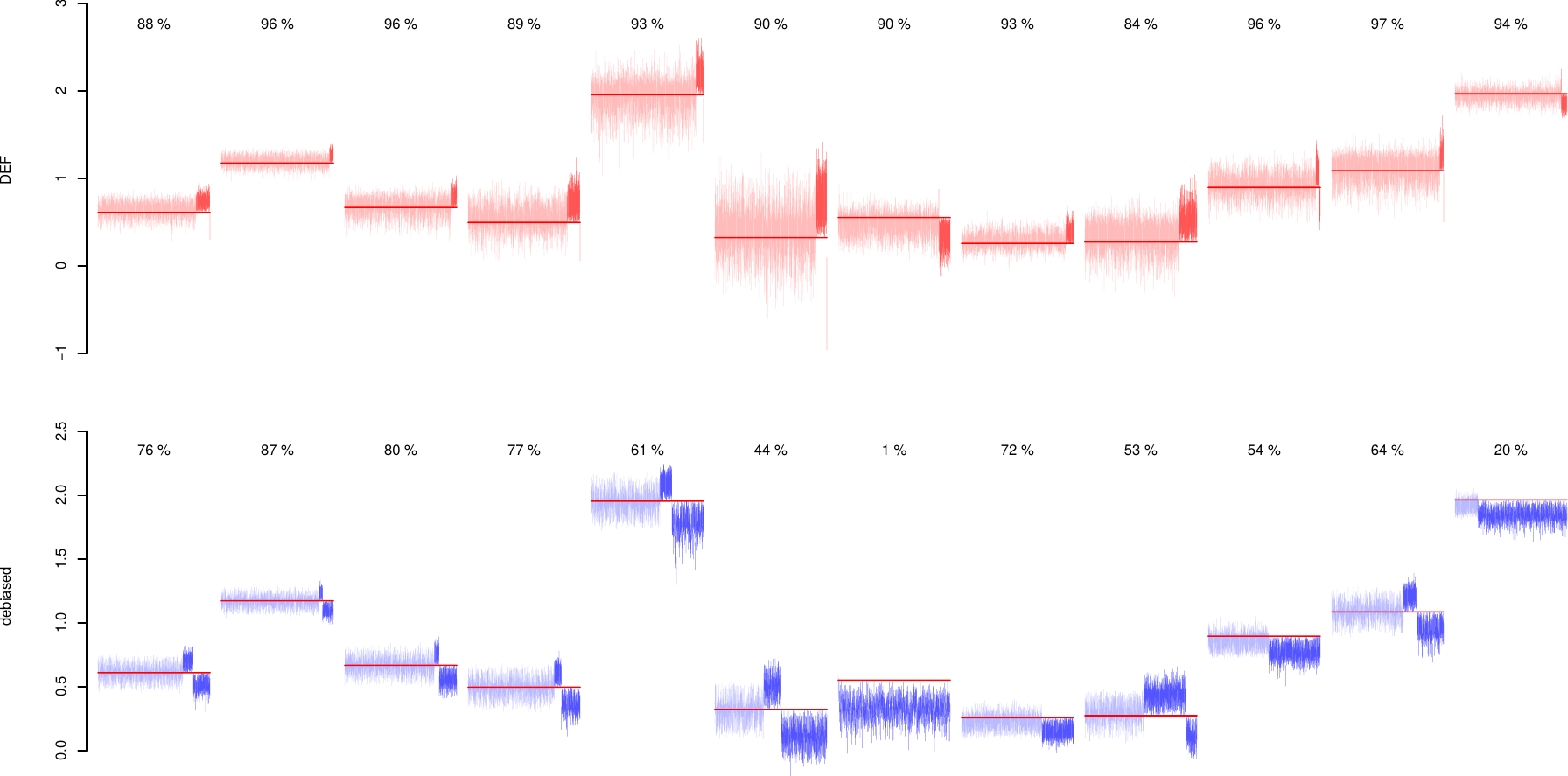}
			\caption{The highly nonlinear setting (c); the interpretation is similar to that of Figure \ref{fig:n1_real_lin}.\label{fig:n1_real_nonlin2}}
		\end{figure*}
		
		\subsection{Generalised linear models} \label{sec:exp_gen}
		Here we present some simple experiments to investigate the performance of the DEF statistic $T_{\GENDEF}$ for generalised linear models (Section~\ref{sec:high_gen}) where we take the $X$ and $Y$-models to be logistic regression models. We generate data $(Y_i, X_i, Z_i) \in \{0,1\} \times \{0,1\} \times \R^p$ for $i=1,\ldots,n$ with $(n,p)=(250, 100)$ in the following way. We first construct a multivariate distribution $P$ as in Section~\ref{sec:exp_lin}, but take $p=250$. We then simulate $Z_i \iid P$, and independently generate $Y_i \sim \text{Bern}(\pi^Y_i)$ and $X_i \sim  \text{Bern}(\pi^X_i)$ where probabilities $\pi^Y_i$ and $\pi^X_i$ satisfy
		\begin{align}
			\text{logit}(\pi^Y_i) &= \sum_{j=1}^{24} a_j Z_{ij} \beta_j  \label{eq:y_gen_hd}\\
			\text{logit}(\pi^X_i) &= \sum_{j=1}^{4} a_j Z_{ij} \beta_j, \notag
		\end{align}
		with $\beta_j \iid U[0,1]$ and $a_j = 1 - (j-1)/24$. Note that $X_i \ci Y_i \,|\, Z_i$; however the $X_i$ and $Y_i$ are positively correlated, making control of the type I error when performing the conditional independence test challenging.
		
		We generate $6$ sets of $(\beta, \mb Z)$ pairs, and for each of these simulate $250$ realisations of $\mb X$ and $\mb Y$. To each of the $6 \times 250$ datasets, we apply our DEF methodology positing logistic regression models for the $X$ and $Y$-models, and also the debiased Lasso for generalised linear models via weighted least squares (see~Section 3.2 of \citep{dezeure2015high}). The results are given in the top plot of Figure~\ref{fig:high_gen}. We see that that DEF approach is able to control the type I error by exploiting the fact that the $X$-model, being highly sparse, is relatively easy to estimate. On the other hand, the debiased Lasso requires accurate estimation of all 24 components of $\beta$ in the $Y$-model, and as a consequence is highly anti-conservative here.
		
		To assess the power of the methods, we consider an identical setup as just described, but $X_i$ is added to the right-hand side of \eqref{eq:y_gen_hd} to induce dependence. The bottom plot in Figure~\ref{fig:high_gen} presents the corresponding results. We see that whilst the $p$-values for $T_{\GENDEF}$ are sub-uniform, power is reduced compared to the debiased Lasso as expected; this is the price of the additional robustness offered by the DEF approach. 
		\begin{figure} 
			\centering
			\includegraphics[width=0.47\textwidth]{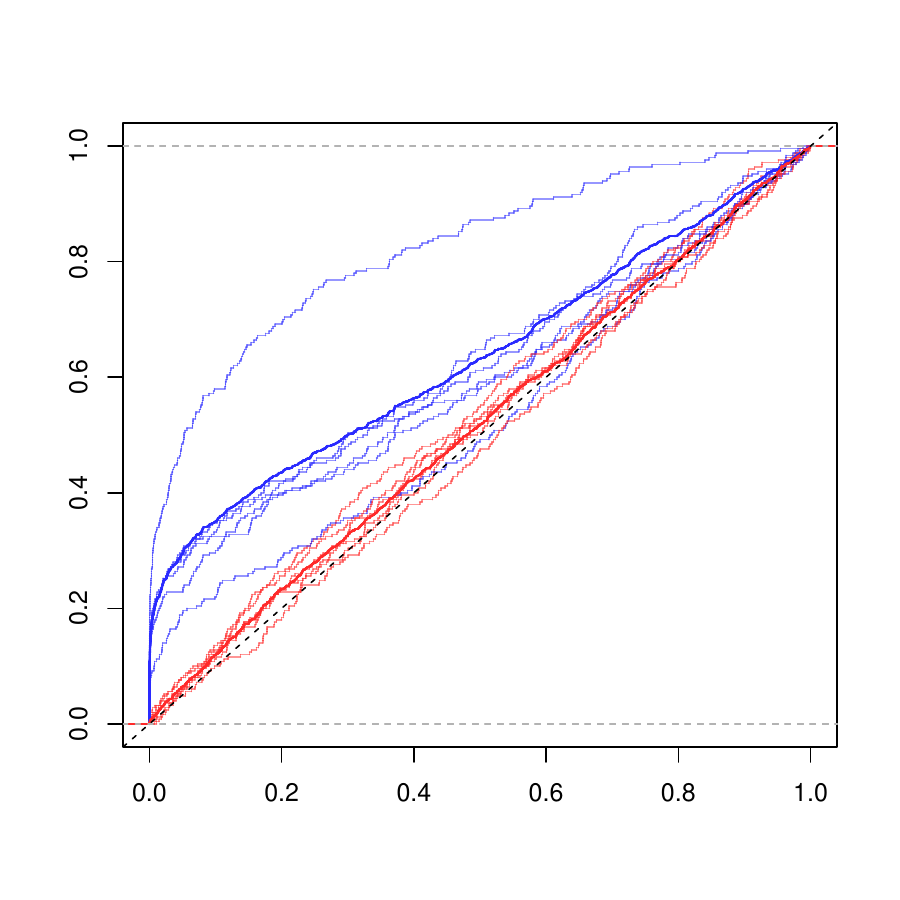}
			\includegraphics[width=0.47\textwidth]{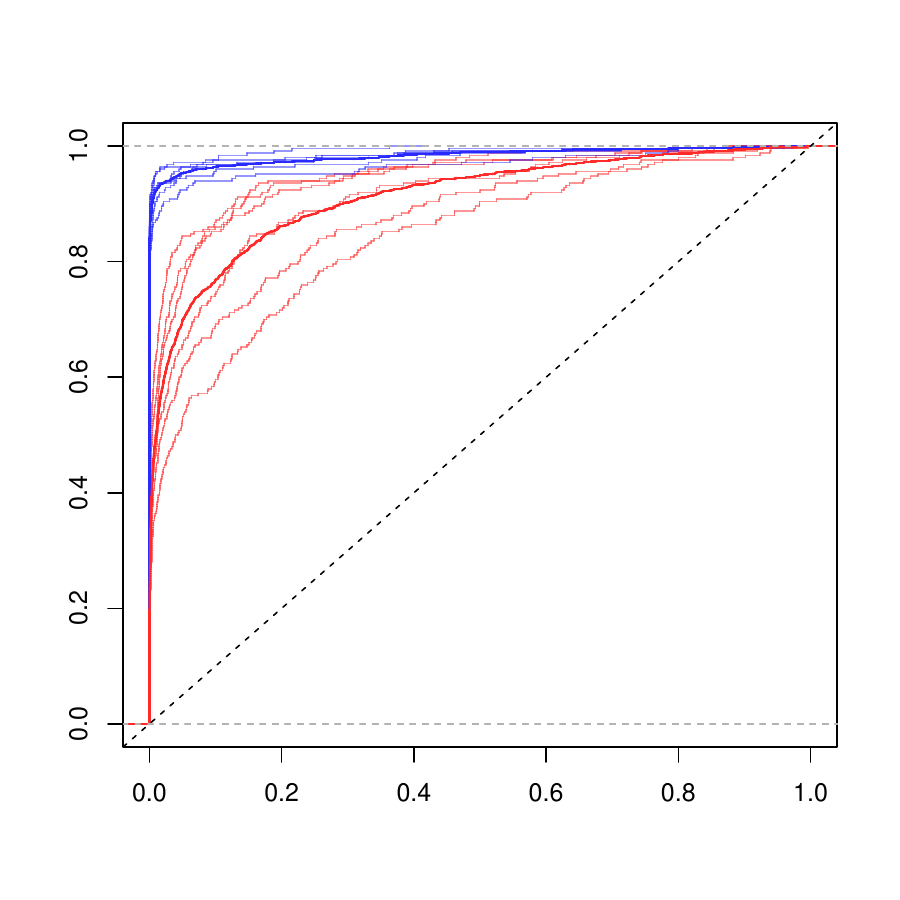}
			\caption{Empirical distribution functions (ECDFs) of $p$-values constructed via the DEF (red) and debiased Lasso (blue) approaches for null (top) and alternative (bottom) settings described in Section~\ref{sec:exp_gen}. In each panel, the fainter and thinner lines correspond to the $6$ setups with different $(\beta, \mb Z)$ whilst the thicker solid lines are aggregate ECDFs.\label{fig:high_gen}}
		\end{figure}
		
		\section{Discussion} \label{sec:discuss}
		In recent years, there has been growing interest in understanding the performance of statistical procedures when the models they have been designed for are misspecified; see for example \citet{bujamodelsI,bujamodelsII}.
		In this work, we consider regression models with response $Y$, a single predictor of interest $X$, and additional covariates $Z \in \R^p$. Our goal is assessing the significance of $X$ after controlling for $Z$, a problem which may be equivalently framed as testing for the null hypothesis $H_0$ of conditional independence $Y \ci X \,|\, Z$. If either the $Y$ or the $X$-model is
		linear or generalised linear, the situation is favourable for DEF inference.
		
		The DEF property holds for a test statistic $T$ if the following is
		true. Under $H_0: X \ci Y \,|\,Z$ we have $T
		\indist {\mathcal N}(0,1)$ when at least one among the $Y$ and $X$-model
		holds. Examples of such test statistics include the
		following ones:
		\begin{enumerate}[(i)]
			\item $T_{\OLS}$, the standard $t$-statistic  for
			testing 
			significance of the parameter corresponding to $X$ as laid out in
			Theorem~\ref{thm:lm};
			\item $T_{\GLM, 1}$ and $T_{\GLM, 2}$, the modified Wald statistics with correction
			factors (see $\eqref{eq:correction}$) for the standard error as
			discussed in Section~\ref{sec:low_gen};
			\item $T_{\DEF}$ in \eqref{eq:T_DEF} based on a symmetrised version
			of the debiased Lasso in a high-dimensional linear model as
			discussed in  Theorem~\ref{thm:DEF_lin};
			\item $T_{\GENDEF}$ based on a symmetrised version of the debiased
			Lasso in high-dimensional generalised linear models
			%a high-dimensional weighted linear model
			as discussed
			in Section~\ref{sec:high_gen}.
		\end{enumerate}
		In cases (iii) and (iv), we explicitly model both the $X$ and $Y$
		regressions, and also explicitly build in symmetry into the test statistics
		to reflect the symmetry of the null hypothesis. On the other hand, the first two examples,
		which relate to low-dimensional settings, are not obviously engineered to
		have the DEF property. An interesting finding here is that the these
		classical test statistics implicitly use a linear $X$-model. We may
		speculate that this hidden robustness of classical significance tests to
		potentially severe $Y$-model misspecification has in some way contributed
		to their popularity and usefulness given that
		all models---but as we have established here, \emph{not}  all inferential
		tools---are wrong \citep{box1976science}.
		
		As a separate point of interest, we argue that confidence intervals in
		high-dimensional settings should be constructed via inversion of tests
		instead of relying directly on asymptotic distribution theory for the relevant pivot. Supporting empirical evidence
		is given in Section~\ref{sec:high_conf}.
		
		Our work also offers a number of potentially fruitful directions for
		further research. For example, it would be interesting to investigate the power
		properties of our DEF procedures. In addition, lower bounds on the power
		that can be achieved subject to a DEF property holding would be worth
		exploring. Finally, the analogue of the method proposed for confidence
		interval construction via inverting tests seems not to have the DEF
		property in the context of generalised linear models. It would be very
		useful to develop DEF confidence intervals for this setting, or indeed
		prove that it is in some sense not possible.  

%%%%%%%%%%%%%%%%%%%%%%%%%%%%%%%%%%%%%%%%%%%%%%
%% Single Appendix:                         %%
%%%%%%%%%%%%%%%%%%%%%%%%%%%%%%%%%%%%%%%%%%%%%%
%\begin{appendix}
%\section*{???}%% if no title is needed, leave empty \section*{}.
%\end{appendix}
%%%%%%%%%%%%%%%%%%%%%%%%%%%%%%%%%%%%%%%%%%%%%%
%% Multiple Appendixes:                     %%
%%%%%%%%%%%%%%%%%%%%%%%%%%%%%%%%%%%%%%%%%%%%%%
\begin{appendix}
\section{Proofs}
\subsection{Proof of Theorem~\ref{thm:lm}} \label{sec:proof_prop1}
The relationship \eqref{eq:parcor_rel} between the $t$-statistic $T$ and the partial correlation $\hat{\rho}$ follows easily from the following observations:
\begin{gather*}
	\hat{\theta} = \frac{\mb Y^T (\mb I - \mb P) \mb X}{\|(\mb I - \mb P) \mb X\|_2^2}, \qquad \{(\tilde{\mb Z}^T \tilde{\mb Z})^{-1}\}_{11} = \|(\mb I - \mb P) \mb X\|_2^{-2}, \\
	\begin{split}
			\|(\mb I - \tilde{\mb P})\mb Y\|_2^2 &= \|(\mb I - \mb P)\mb Y\|_2^2 - \frac{\{\mb Y^T(\mb I - \mb P)\mb X\}^2}{\|(\mb I - \mb P)\mb X\|_2^2} \\
			&= \|(\mb I - \mb P)\mb Y\|_2^2(1-\hat{\rho}^2).
	\end{split}
\end{gather*}
Thus it suffices to show that $\sqrt{n} \hat{\rho} \indist \mathcal{N}(0, 1)$ since this implies that $\hat{\rho} \inprob 0$.
As $\hat{\rho}$ is symmetric in $\mb X$ and $\mb Y$, we need only show these facts hold assuming (Y1) and (Y2). Note then we have
\[
\hat{\rho} =  \frac{\mb X^T(\mb I - \mb P)\mbb\varepsilon}{\|(\mb I - \mb P)\mbb\varepsilon\|_2 \|(\mb I - \mb P)\mb X\|_2}
\]
where since $\mb X \ci \mb Y \,|\, \mb Z$, the properties of $(\varepsilon_i)_{i=1}^n$ hold conditionally on $(\mb Z, \mb X)$.
We first show $\|(\mb I - \mb P)\mbb\varepsilon\|_2/\sqrt{n} \inprob \sigma$. We have
\begin{align} \label{eq:sigma_tilde_sq}
	\frac{1}{n}\|(\mb I - \mb P)\mbb\varepsilon\|_2^2 = \frac{1}{n}\|\mbb\varepsilon\|_2^2 - \frac{1}{n} \mbb\varepsilon^T \mb P \mbb\varepsilon.
\end{align}
By the weak law of large numbers, the first term converges in probability to $\sigma^2$. For the second term, note that due to (Y1), using the cyclic property of the trace operator,
\begin{align*}
\E  \mbb\varepsilon^T \mb P \mbb\varepsilon & = \E  \tr(\mbb\varepsilon^T \mb P \mbb\varepsilon ) = \tr (\E  \mbb\varepsilon \mbb\varepsilon^T \mb P) \\
 = & \tr \E\{\underbrace{\E(  \mbb\varepsilon \mbb\varepsilon^T  \,|\, \mb P)}_{=\sigma^2 \mb I} \mb P\}] = \sigma^2 \tr(\mb P) \leq \sigma^2 p.
\end{align*}
Thus the final term in \eqref{eq:sigma_tilde_sq} has expectation tending to $0$ as $p/n \to 0$. By Markov's inequality, this must therefore go to 0 in probability, and so $\|(\mb I - \mb P)\mbb\varepsilon\|_2/\sqrt{n} \inprob \sigma$ as required.

Next we claim that
\begin{equation*}
	A_n := \frac{\mb X^T (\mb I - \mb P)\mbb\varepsilon}{\|(\mb I - \mb P)\mb X\|_2}  = \|\mb R\|_2^{-1} \sum_{i=1}^n R_i\varepsilon_i \indist \mathcal{N}(0, \sigma^2).
\end{equation*}
Note that conditional on $(\mb X, \mb Z)$, the $\varepsilon_i$ are i.i.d.\ with variance $\sigma^2$ and third moment bounded by $M$.
Lemma~\ref{lem:CLT} below with $\mb R_n = (\mb I - \mb P) \mb X$ then shows that $A_n \indist \mathcal{N}(0, \sigma^2)$.
%Thus
%by the Berry--Esseen theorem,
%\[
%|\pr(A_n \leq t | \mb X, \mb Z) - \Phi(t)| \leq \frac{cM^{3/4}}{\sigma^3} \frac{\sum_{i=1}^n |R_i|^3}{\Big(\sum_{i=1}^n R_i^2 \Big)^{3/2}}
%\]
%for all $t$, where $c$ is some universal constant. But the LHS is at least
%\[
%|\E\{\pr(A_n \leq t | \mb X, \mb Z)\} - \Phi(t)| = |\pr(A_n \leq t) - \Phi(t)|.
%\]
Combining with the previous result and applying Slutsky's Lemma gives $\sqrt{n}\hat{\rho} \indist \mathcal{N}(0, 1)$ as required

%It remains to show $\hat{\rho} \inprob 0$ for which it suffices to show $\E\Var(A_n | \mb X, \mb Z) / n \to 0$. This latter quantity is equal to $\sigma^2 / n \to 0$.
\qed

\begin{lem} \label{lem:CLT}
	Let $(\varepsilon_{in})_{i \leq n}$ and $(R_{in})_{i \leq n}$ be triangular arrays of random variables and define $\mb R_{n} = (R_{1n}, \ldots, R_{nn})$ for all $n$. Assume these random variables satisfy the following conditions:
	\begin{enumerate}[(i)]
		\item $\varepsilon_{1n}, \ldots, \varepsilon_{nn}$ are independent conditional on $\mb R_n$;
		\item for all $i=1,\ldots,n$ and some $\delta, M > 0$,
		\begin{gather*}
		\E(\varepsilon_{in} | \mb R_n)=0, \; \; \E(\varepsilon_{in}^2 | \mb R_n) = \sigma^2>0, \\ \E(|\varepsilon_{in}|^{2+\delta} | \mb R_n) < M;
		\end{gather*}
		\item $\pr(\mb R_n =\mb 0) \to 0$;
		\item for some $\delta > 0$,
		\begin{equation*}
			A_n := \begin{cases}
				\frac{1}{\|\mb R_n\|_2^{2+\delta}} \sum_{i=1}^n |R_{in}|^{2+\delta} & \quad \text{ if } \mb R_n \neq \mb 0, \\
				0 & \quad \text{ if } \mb R = \mb 0,
			\end{cases}
		\end{equation*}
	satisfies$A_n \inprob 0$, and
		\[
		\frac{1}{\|\mb R_n\|_2^{2+\delta}} \sum_{i=1}^n |R_{in}|^{2+\delta} \ind_{\{\mb R_n \neq \mb 0\}} \inprob 0.
		\]
		%\[
		%\frac{1}{\sqrt{n}} \E \Bigg( \frac{\frac{1}{n}\sum_{i=1}^n |R_{in}|^3}{\Big(\frac{1}{n}\sum_{i=1}^n R_{in}^2\Big)^{3/2}} \ind_{\{\mb R_n \neq 0\}}\Bigg) \to 0.
		%\]
	\end{enumerate}
	Then 
	\[
	B_n := \|\mb R_n\|_2^{-1} \sum_{i=1}^n R_{in} \ind_{\{\mb R_n \neq 0\}} \varepsilon_{in} \indist \mathcal{N}(0, \sigma^2).
	\]
\end{lem}
\begin{proof}
	Let the random sequences above be defined on a common probability space $(\Omega, \mathcal{F}, \pr)$.
	Let $(n_k)_{k=1}^\infty \subseteq \mathbb{N}$ be an arbitrary subsequence. Then we know there exists a further subsequence $(n_{k(l)})_{l=1}^\infty$ on which the following occur:
	\begin{enumerate}[(a)]
		\item the convergence in (iv) above happens almost surely, that is the probability that
		\[
		\lim_{l \to \infty}  \frac{1}{\|\mb R_{n_{k(l)}}\|_2^{2+\delta}} \sum_{i=1}^{n_{k(l)}} |R_{in_{k(l)}}|^{2+\delta} \ind_{\{\mb R_{n_{k(l)}} \neq \mb 0\}} = 0
		\]
		equals one.
%		\[
%		\pr\bigg( \lim_{l \to \infty}  \frac{1}{\|\mb R_{n_{k(l)}}\|_2^{2+\delta}} \sum_{i=1}^{n_{k(l)}} |R_{in_{k(l)}}|^{2+\delta} \ind_{\{\mb R_{n_{k(l)}} \neq \mb 0\}} = 0 \bigg) = 1;
%		\]
		\item $\sum_{l=1}^\infty \pr(\mb R_{n_{k(l)}} = \mb 0) < \infty$.
	\end{enumerate}
	By the first Borel--Cantelli lemma, we have that the sequence of events $\Omega_l := \mb R_{n_{k(l)}} \neq \mb 0$ satisfies $\pr(\liminf_{l \to \infty} \Omega_l)=1$. Let $\Omega_2$ be the intersection of the event in (a) above and $\liminf_{l \to \infty} \Omega_l$. Note that $\pr(\Omega_2)=1$.
	%Further define for each $n$, $\mathcal{R}_n := \{\mb R_n(\omega) : \omega \in \Omega_2\} \subseteq \R^n$.
	%
	%	Fix $n$ and write $\mbb\varepsilon := (\varepsilon_{in})_{i=1}^n$. Consider now the collection of probability measures $\{\pr_{\omega,n}\}_{\omega \in \Omega}$ on $(\Omega, \sigma(\mbb\varepsilon))$ given by
	%	\[
	%	\pr_{\omega,n} \left( \mbb\varepsilon^{-1} (B) \right) = \E \left\{\ind_B\left(\mbb\varepsilon \right) \,|\, \mb R_{n} \right\}(\omega)
	%	\]
	%	for each Borel set $B \subseteq \R^{n}$.
	
	Now observe that for each $\omega \in \Omega_2$, writing $\mb r:=\mb R_n(\omega)$, we have $\mb r \neq \mb 0$ and
	\begin{align*}
		C(\omega,n):&=\frac{\sum_{i=1}^{n} \E(|R_{in} \varepsilon_{in}|^{2+\delta} \,|\, \mb R_n=\mb r)}{\left(\sum_{i=1}^n \E (R_{in}^2 \varepsilon_{in}^2 \,|\, \mb R_n=\mb r) \right)^{1 + \delta/2}} \\
		&=
		\frac{\sum_{i=1}^{n} |r_{in}|^{2+\delta}\E (|\varepsilon_{in}|^{2+\delta} \,|\, \mb R_n=\mb r)}{\left(\sum_{i=1}^n r_{in}^2 \E(\varepsilon_{in}^2 \,|\, \mb R_n= \mb r) \right)^{1 + \delta/2}} \\
		&< \frac{M}{\sigma^2}\frac{\sum_{i=1}^{n} |r_{in}|^{2+\delta}}{\left(\sum_{i=1}^n r_{in}^2 \right)^{1 + \delta/2}}.
	\end{align*}
	Thus
	\[
	\lim_{l \to \infty} B(\omega, n_{k(l)}) = 0
	\]
	for all $\omega \in \Omega_2$.
	
	For each $n$, let $\tilde{P}_n : \R^n \times \mathcal{F} \to [0,1]$ be a regular conditional probability given $\mb R_n$, and for $\omega \in \Omega$, let $P_{n,\omega} : \mathcal{F} \to [0,1]$ be given by $P_{n,\omega}(A) = \tilde{P}_n(\mb{R}_n(\omega),A)$. Denoting expectation with respect to $P_{n,\omega}$ by $E_{n,\omega}$, note that
	\[
	C(\omega,n) = \frac{\sum_{i=1}^{n} E_{n,\omega}( |R_{in} \varepsilon_{in}|^{2+\delta})}{\left(\sum_{i=1}^n E_{n,\omega}(R_{in}^2 \varepsilon_{in}^2) \right)^{1 + \delta/2}}.
	\]
	%Note that for any continuous bounded function $g:\R \to \R$ and $\omega \in \Omega_2$,
	%\[
	%\int_\Omega g(A_{n_{k(l)}}(\omega')) dP_n(\omega, \omega') = \int_\Omega g(A_{n_{k(l)}}(\omega')) dP_n(\omega, \omega')
	%\]
	From the above, for each $\omega \in \Omega_2$, we can apply the Lindeberg--Feller central limit theorem for triangular arrays \citep[Prop.~2.27]{van2000asymptotic} along the sequence of probability measures given by $(P_{n_{k(l)},\omega})_{l=1}^\infty$, noting that Lyapunov's condition implies the Lindeberg--Feller condition. Writing $Z \sim \mathcal{N}(0, \sigma^2)$, we have that for any $\omega \in \Omega_2$ and
	any continuous bounded function $g:\R \to \R$,
	\begin{align*}
		\lim_{l \to \infty} E_{n_{k(l)},\omega}\left\{ g\left(B_{n_{k(l)}}\right)\right\} &=
		\lim_{l \to \infty} \E\left\{ g\left(B_{n_{k(l)}}\right) \,|\, \mb R_{n_{k(l)}} \right\}(\omega) \\
		&= \E g(Z).
	\end{align*}
	Now as $\pr(\Omega_2)=1$, we have
	\[
	\lim_{l \to \infty} \E \left\{ g\left(B_{n_{k(l)}}\right) \,|\, \mb R_{n_{(k(l)}} \right\} = \E g(Z)
	\]
	almost surely.
	Then as the subsequence $(n_k)_{k=1}^\infty \subseteq \mathbb{N}$ was arbitrary, we see that in fact
	\[
	\E \left\{ g\left(B_{n}\right) \,|\, \mb R_n \right\} \inprob \E g(Z).
	\]
	Finally, note that as $g$ is bounded, we may apply dominated convergence theorem to show that
	\[
	\E \left\{ g\left(B_{n}\right)\right\} \to \E g(Z).
	\]
	As this holds for every continuous bounded $g$, we have the result.
\end{proof}

\subsection{Proof and regularity conditions for Theorem~\ref{thm:gen_lin}} 
\subsubsection{Regularity conditions} \label{sec:reg1}
Assume the following regularity conditions.
\begin{enumerate}[(i)]
	\item $L(x t + z^T\beta;y) > 0$ almost everywhere  and $\E | \ell(X t +Z^T\beta;Y)| < \infty$ for all $(t, \beta) \in \Theta$.
	\item The ratio $L(xt_1 + z^T\beta_1;y) / L(xt_2 + z^T\beta_2;y)$ is not almost everywhere equal to $1$ when $(t_1, \beta_1) \neq (t_2, \beta_2)$.
	\item There exists an open set $K \subset \Theta$ containing $(0, \beta^\dagger)$ such that $\E\ell(Xt  + Z^T\beta)$ is partially differentiable with respect to $t$ and with respect to $\beta_j$ for all $j$, with  integrable derivatives given by $\E\{X U(Xt + Z^T\beta;Y)\}$ and $\E\{Z_jU(Xt + Z^T\beta;Y)\}$ respectively.
	%\item 
	%	The exists an open set $K \subset \Theta$ containing $(0, \beta^\dagger)$ satisfying the following: there exists measurable $e: \R^{p+2}$ with $\E e(X, Y, Z) <\infty$ satisfying for all $(t, \beta) \in K$, 
	%	\[
	%	|XL'(Xt + Z^T\beta;Y)| + \|Z L'(Xt + Z^T\beta;Y)\|_2 \leq e(X, Y, Z).
	%	\]
	%	\item $\E X^2 < \infty$ and $\E U^2(Z^T\beta^{\dagger};Y) < \infty$.
\end{enumerate}
%Need to show that $(0, \beta^Y)$ is the minimiser.
%Suppose that both $(0, \beta^{\dagger})$ and
%\[
%\argmin
%\]
%are uniquely defined and lie in the
%Suppose that for some open set $K \subset \R^{1+p}$ containing $(0, \beta^\dagger)$, we have
%\[
%\int_y \sup_{(\theta, \beta) \in K} L'(x\theta + z^T\beta;y) d\mu(y) < \infty
%\]
%for all $x$ and $z$ for which the marginal density $f_{XZ}$ of $(X, Z)$ (with respect to $\mu$) has $f_{XZ}(x, z)>0$. We further assume that $\E \{U^2(Z^T\beta^\dagger;Y)\} < \infty$ and $\E (X^2) < \infty$.
\subsubsection{Proof of Theorem~\ref{thm:gen_lin}}
Suppose first that the $Y$-model is well-specified. Then as $X \ci Y \,|\, Z$, we know from (ii) that $\theta=0$. Standard arguments show that then $(0, \beta^Y)$ maximises $\E \ell(t X +Z^T\beta;Y)$ over $(t, \beta) \in \Theta$ and satisfies the score equations. Thus $\beta^Y = \beta^\dagger$.

Let us now consider the case where
the $X$-model is linear. We first show that $(\theta, \beta)=(0, \beta^\dagger)$ satisfies \eqref{eq:z_score}.
%  The regularity conditions allow for exchanging differentiation and expectation to give that for all $(\theta, \beta) \in K$, $\E\ell(X\theta + Z^T\beta)$ is partially differentiable with respect to $\beta_j$ for all $j$, with derivatives given by $\E\{ZU(X\theta + Z^T\beta;Y)\}$.
By optimality of $\beta^\dagger$, we must have
\begin{equation} \label{eq:partial_d}
	\E\{ZU(Z^T\beta^\dagger;Y)\}=0,
\end{equation}
so $(t, \beta)=(0, \beta^\dagger)$ satisfies \eqref{eq:z_score}. It suffices to check that this also satisfies \eqref{eq:x_score}. We have
\begin{align}
	\E\{X U(Z^T\beta^\dagger;Y)\} &= \E[ \E\{X U(Z^T\beta^\dagger;Y)\,|\,Z\}] \notag\\
	&= \E[\E\{(Z^T\beta^X + \varepsilon ) U(Z^T\beta^\dagger;Y)\,|\,Z\}] \notag\\
	&= \E[\E\{\varepsilon U(Z^T\beta^\dagger;Y)\,|\,Z, Y\}], \label{eq:cond_exp0} \\
	&= \E[U(Z^T\beta^\dagger;Y) \E\{\varepsilon \,|\,Z, Y\}] = 0.
\end{align}
using property \eqref{eq:partial_d} of $\beta^\dagger$ in \eqref{eq:cond_exp0} and that $\E\{\varepsilon \,|\,Z, Y\} = \E\{\varepsilon \,|\,Z\}$ due to the conditional independence $X \ci Y \,|\, Z$ in the final line. \qed
%We now appeal to Daudin's characterisation of conditional independence \citep{Daudin1980}, that $X \ci Y \,|\, Z$ if and only if for each $f(X, Z)$ and $g(Y, Z)$ with $\E f^2(X, Z) < \infty$ and $\E g^2(X, Z) < \infty$, we have
%\[
%\E\big( [f(X, Z) - \E\{ f(X, Z)\,|\,Z\}] g(Y, Z) \big) = 0.
%\]
%We apply this with $f(X, Z) \equiv \varepsilon$ and $g(Y, Z) \equiv U(Z^T\beta^\dagger;Y)$; note that $\E(\varepsilon^2) \leq \E(X^2)<\infty$. We thus have that \eqref{eq:cond_exp0} is equal to $0$ as required.
\subsection{Proof and regularity conditions for Theorem~\ref{thm:gen_var}} 
\subsubsection{Regularity conditions} \label{sec:reg2}
In addition to the regularity conditions laid out in Section~\ref{sec:reg1}, we assume the following.
\begin{enumerate}[(i)]
	\item When the $Y$-model holds, differentiation and integration can be interchanged such that the variance of the score is equal to the Fisher information matrix, i.e., $H=V$, and moreover $\E L''(Z^T\beta^*;Y)=0$.
	\item The solution $(\theta^*, \beta^*)$ to the score equations is unique.
	\item $\E U'(Z^T\beta^*) \neq 0$.
	\item $ \E |U'(Z^T\beta^*)| <\infty$, $\E \|Z\|_2^2 |U'(Z^T\beta^*)| < \infty$ and  $\E \{X^2 |U'(Z^T\beta^*)|\} < \infty$.
\end{enumerate}
%Suppose that for each $(x, z)$ for which $f_{XZ}(x, z)>0$
%\[
%\int_y \sup_{(\theta, \beta) \in K} L''(x\theta + z^T\beta;y) d\mu(y) < \infty.
%\]
\subsubsection{Proof of Theorem~\ref{thm:gen_var}}
When the $Y$-model holds, we have $H^{-1}VH^{-1}=H^{-1}$, and
\[
\frac{\E\{U^2(Z^T\beta^*)\}}{\E\{ U'(Z^T\beta^*)\}}=1
\]
as $\E L''(Z^T\beta^*;Y)=0$.
We now turn to the case where the $X$-model holds. Let $\varepsilon = X - Z^T\beta^X$ and note that $\E(\varepsilon\,|\,Z)=0$.
We know from Theorem~\ref{thm:gen_lin} that $\theta^*=0$. Let us first compute $H$. We have 
\[
-H_{1,j+1} = \E \{X Z_j U'(Z^T\beta^*)\}.
\]
Now 
\begin{align*}
	& \E \{X Z_j U'(Z^T\beta^*) \,|\, Y, Z \} \\
	= & Z_j U'(Z^T\beta^*) \E(Z^T \beta^X + \varepsilon \,|\,Y, Z) \\
	= &Z_j U'(Z^T\beta^*) Z^T \beta^X.
\end{align*}
Here we have used the fact that as $Y \ci Y \,|\, Z$, $\E(\varepsilon \,|\, Y, Z) = \E( \varepsilon \,|\, Z) = 0$. Considering now $H_{11}$, we have
\[
\E\{X^2 U'(Z^T\beta^*) |Y, Z\} = \{\E(\varepsilon^2|Y,Z) + (Z^T \beta^X)^2\} U'(Z^T\beta^*).
\]
Thus, writing $A = \E \{ZZ^T U'(Z^T\beta^*)\} \in \R^{p \times p}$, we have
\[
H = -\begin{pmatrix} (\beta^X)^T A\beta^X + \E\{\varepsilon^2U'(Z^T\beta^*)\}  & (\beta^X)^T A  \\
	A\beta^X & A \end{pmatrix}.
\]
Using standard formulas for the blockwise inverse of matrices in terms of Schur complements, we have that the first column $h$ of $H^{-1}$ satisfies
\[
h= \begin{pmatrix} -1 \\  \beta^X \end{pmatrix} [\E\{\varepsilon^2U'(Z^T\beta^*)\}]^{-1}.
\]
Thus
\begin{align} \label{eq:mod_var}
	(H^{-1} V H^{-1})_{11} = h^T V h= \frac{\E\{ \varepsilon^2 U^2 (Z^T \beta^*)\}}{ [\E\{\varepsilon^2  U'(Z^T\beta^*)\}]^{2}}.
\end{align}
Now as $\E(\varepsilon^2 \,|\,Z) = \Var(X\,|\,Z)=\Var(X)=\E(\varepsilon^2)$, we have that for any measurable function $f$ of $Z$ with $\E |f(Z)| \,, \E (|f(Z)| \varepsilon^2) < \infty$,
\[
\E\{\varepsilon^2 f(Z)\} = \E [f(Z)\E\{ \varepsilon^2 \,|\, Z\}] = \E\{f(Z)\} \E\{\varepsilon^2\}.
\]
Thus we have that the quantity in \eqref{eq:mod_var} is equal to
\begin{align*}
	\frac{\E \{ U^2 (Z^T \beta^*)\}}{\E (\varepsilon^2) \, \{\E\,  U'(Z^T\beta^*)\}^{2}}  = - (H^{-1})_{11} \frac{\E\{U^2(Z^T\beta^*)\}}{\E\{ U'(Z^T\beta^*)\}}.
\end{align*}
\qed
\subsection{Proof and regularity conditions for Theorem~\ref{thm:wald}}

\subsubsection{Regularity conditions} \label{sec:reg3}
In addition to the regularity conditions laid out in Sections~\ref{sec:reg1} and \ref{sec:reg2}, we  assume that $\Theta$ is compact and that there exists functions $f_1,f_2 : \R^{p+2} \to [0, \infty)$ such that for all $(t,\beta) \in \Theta$,
\begin{align*}
	|U'(tX + Z^T\beta;Y)| &\leq f_1(X,Y,Z), \\
	U^2(tX + Z^T\beta;Y) &\leq f_2(X,Y,Z)
\end{align*}
with $\E f_j(X,Y,Z) < \infty$ for $j=1,2$. We further assume that $\E U^2(Z^T\beta^*) > 0$ and that $U'$ is continuous.
\subsubsection{Proof of Theorem~\ref{thm:wald}}
From Theorems~\ref{thm:gen_lin} and \ref{thm:gen_var}, it suffices by Slutsky's lemma and the continuous mapping theorem to show that
\[
\hat{C}_j \inprob \frac{\E\{U^2(Z^T\beta^*)\}}{\E\{ U'(Z^T\beta^*)\}}.
\]
Let us consider $j=2$; the arguments are similar for $j=1$. By Slustky's lemma, it suffices to show that
\begin{align}
	\frac{1}{n}\sum_{i=1}^n U'(\hat{\theta}X_i + Z_i^T\beta) &\inprob \E U'(tX + Z^T\beta)  \label{eq:prob1}\\
	\frac{1}{n}\sum_{i=1}^n U^2(\hat{\theta}X_i + Z_i^T\beta) &\inprob \E U^2(tX + Z^T\beta). \label{eq:prob2}
\end{align}
Theorem 2 of \citet{jennrich1969asymptotic} shows that 
\begin{align*}
	\sup_{(\beta,t) \in \Theta} \abs{\frac{1}{n}\sum_{i=1}^n U'(tX_i + Z_i^T\beta) - \E U'(tX + Z^T\beta)} &\to 0 \\
	\sup_{(\beta,t) \in \Theta} \abs{\frac{1}{n}\sum_{i=1}^n U^2(tX_i + Z_i^T\beta) - \E U^2(tX + Z^T\beta)} &\to 0
\end{align*}
almost surely.
By assumption, $(\hat{\theta}, \hat{\beta}^Y) \inprob  (\theta^*, \beta^*) = (0, \beta^*)$, using Theorem~\ref{thm:gen_lin} for the final equality. Thus for any subsequence $(m(n))_{n=1}^\infty$, there exists a further subsequence $(l_{m(n)})_{n=1}^\infty$ on which the above convergence is almost sure. Let us write $f(t, \beta) = \E U'(tX + Z^T\beta)$. Then given $\epsilon >0$, there exists $N_1$ such that for all $n \geq N_1$,
\begin{align*}
	& \abs{\frac{1}{l_{m(n)}}\sum_{i=1}^{l_{m(n)}} U'(\hat{\theta}X_i + Z_i^T\hat{\beta}^Y) - f(\hat{\theta},\hat{\beta}^Y)} \\
	\leq &\sup_{(\beta,t) \in \Theta} \abs{\frac{1}{l_{m(n)}}\sum_{i=1}^{l_{m(n)}} U'(tX_i + Z_i^T\beta) - f(t,\beta)} < \frac{\epsilon}{2}.
\end{align*}
Note that $\hat{\theta}$ and $\hat{\beta}^Y$ depend on the sample size, though we have suppressed this in the notation.
Meanwhile, by continuity of $U'$ and the continuous mapping theorem, on along $(l_{m(n)})_{n=1}^\infty$ we have
\[
U'(\hat{\theta}X + Z^T\hat{\beta}^Y) \to U'(Z^T\beta^*)
\]
almost surely. Thus by dominated convergence, we have that
\[
f(\hat{\theta}, \hat{\beta}^Y) \to f(0, \beta^*)
\]
along the same subsequence, and so there exists $N_2 \geq N_1$ such that for all $n \geq N_2$
\[
|f(\hat{\theta}, \hat{\beta}^Y) - f(0, \beta^*)| < \epsilon/2 ;
\]
note that $\hat{\theta}$ and $\hat{\beta}^Y$ above are evaluated at sample sizes $l_{m(n)}$ for $n \geq N_2$. Putting things together, we see that on the subsequence $(l_{m(n)})_{n=1}^\infty$, we have
\[
\frac{1}{l_{m(n)}}\sum_{i=1}^{l_{m(n)}} U'(\hat{\theta}X_i + Z_i^T\hat{\beta}^Y) \to  \E U'(Z^T \hat{\beta}^Y)
\]
almost surely. As the original subsequence $(m(n))_{n=1}^\infty$ was arbitrary, we see that \eqref{eq:prob1} holds. The argument to show \eqref{eq:prob2} proceeds similarly. \qed

\subsection{Proof of Theorem~\ref{thm:DEF_lin}}
By symmetry, it is enough to show the result when (Y1) and (Y3)--(Y6) hold. On the event where $\mb R \neq \mb 0$, we have
\begin{align*}
	\frac{(\mb Y - \mb Z \hat{\beta}^Y)^T(\mb X - \mb Z \hat{\beta}^X)}{\|\mb X - \mb Z \hat{\beta}^X\|_2} &= \frac{\mb R^T}{\|\mb R\|_2} \mb Z (\beta^Y - \hat{\beta}^Y) +  \frac{\mb R^T}{\|\mb R\|_2} \mbb\varepsilon.
\end{align*}
The KKT conditions of the Lasso regression of $\mb X$ on $\mb Z$ imply $\|\mb Z^T \mb R \|_\infty / \|\mb R\|_2 \leq \sqrt{n}\lambda_X$. 
Thus by H\"older's inequality and (Y4), we have that
\begin{align*}
	&|\mb R^T \mb Z (\mbb\beta^Y - \hat{\mbb\beta}^Y)|/\|\mb R\|_2 \ind_{\{\mb R \neq \mb 0\}} \\
	\leq& \|\mb R^T \mb Z \|_\infty \|\mbb\beta^Y - \hat{\mbb\beta}^Y\|_1 / \|\mb R\|_2  \ind_{\{\mb R \neq \mb 0\}}\\
	=& O_{\pr} (\sqrt{\log(p)} \times s_Y \sqrt{\log(p)/n} ).
\end{align*}
From (Y3) we see that $|\mb R^T \mb Z (\mbb\beta^Y - \hat{\mbb\beta}^Y)| / \|\mb R\|_2 \ind_{\{\mb R \neq \mb 0\}} \inprob 0$.

The proof that $\mb R^T \mbb\varepsilon / \|\mb R\|_2 \ind_{\{\mb R \neq \mb 0\}} \indist \mathcal{N}(0, \sigma^2)$ is identical to the argument used in the proof of Theorem~\ref{thm:lm} and uses Lemma~\ref{lem:CLT} (see Section~\ref{sec:proof_prop1}). Slutsky's lemma and (Y5) then yield the desired result. \qed

%\subsection{Computation of confidence intervals} \label{sec:conf_comp}
%
%Several small tricks to speed up computation of confidence intervals:
%Fnd the end points of the confidence intervals by a bisection search.
%Warm start Lasso regression using closest point in the bisection
%search.
%Store Z T Z and use ‘covariance updates’ (Friedman et al., 2010) when
%performing coordinate descent to find Lasso solutions.
%This is also very helpful when computing p-values (DEF or
%desparsified) for all variables.
%p = 500, n = 200. Median time (s):
%Covariance updates
%0.33
%glmnet (Friedman et al.)
%23.7

\subsection{Proof of Theorem~\ref{thm:l_1_bd}}
Let $\check{\mb Y} = \mb D^Y \tilde{\mb Y}$ and let $\check{\mb Z} = (\mb D^Y \mb Z\Lambda^Y, \mb D^X \mb Z\Lambda^X)$. Note that $\check{\mb Y} = \check{\mb Z} \vartheta + \mbb\varepsilon^Y$, where $\vartheta \in \R^{2p}$ with $\vartheta_j = (\beta^Y - \hat{\beta}^Y)_j$ for $j \leq p$ and $\vartheta_j = 0$ for $j>p$. We seek to bound $\|\check{\vartheta}\|_1$ where
\[
\check{\vartheta} \in \argmin_{b\in \R^{2p}} \{\|\check{\mb Y} - \check{\mb Z} b\|_2/\sqrt{n} + \lambda\| b \|_1\}.
\]
Now writing $\check{\sigma} = \|\check{\mb Y} - \check{\mb Z} \check{\vartheta}\|_2/\sqrt{n}$, we have that
\[
\check{\vartheta} \in \argmin_{\vartheta \in \R^{2p}} \{\|\check{\mb Y} - \check{\mb Z} b\|_2^2/(2n) + \lambda\check{\sigma}\| b \|_1\}.
\]
This may be seen from examining the KKT conditions of each of the optimisations, which are identical, and take the form
\[
\frac{1}{n} \check{\mb Z}^T(\check{\mb Y} - \check{\mb Z} \check{\vartheta}) = \lambda \check{\sigma} \nu,
\]
where $\|\nu\|_\infty \leq 1$ and $\nu_ j =\sgn(\check{\vartheta}_j)$ for all $j$ such that $\check{\vartheta}_j \neq 0$. Dotting both sides with $\vartheta - \check{\vartheta}$, we obtain
\begin{equation} \label{eq:basic}
	\frac{1}{n} \|\check{\mb Z}(\vartheta - \check{\vartheta})\|_2^2 + \lambda \check{\sigma} \|\check{\vartheta}\|_1 \leq \lambda \check{\sigma}\|\vartheta\|_1  +\frac{1}{n} \|\vartheta - \check{\vartheta}\|_1 \|\check{\mb Z}^T \mbb\varepsilon^Y\|_\infty
\end{equation}
where we have used H\"older's inequality to bound $|\nu^T \vartheta| \leq \|\nu\|_\infty \|\vartheta\|_1 \leq \|\vartheta\|_1$ and  $|\check{\mb Z}(\vartheta - \check{\vartheta})^T\mbb\varepsilon^Y| \leq \|\vartheta - \check{\vartheta}\|_1 \|\check{\mb Z}^T \mbb\varepsilon^Y\|_\infty$, and also the fact that $\check{\vartheta}^T\nu = \|\check{\vartheta}\|_1$.
We now aim to show that with high probability,
\begin{equation} \label{eq:a_bd}
	\frac{\|\check{\mb Z}^T \mbb\varepsilon^Y\|_\infty}{n\check{\sigma}} < a \lambda
\end{equation}
for a constant $a < 1$, where recall that $\lambda= A \sqrt{2\log(p)/n}$ with $A > 1$. We would then have from \eqref{eq:basic} that on the event in question,
\[
\|\check{\vartheta}\|_1 \leq \|\vartheta\|_1 + a \|\vartheta - \check{\vartheta}\|_1 \leq (1+a)\|\vartheta\|_1 + a\|\check{\vartheta}\|_1, 
\]
by the triangle inequality, whence
\[
\|\check{\vartheta}\|_1 \leq \frac{1+a}{1-a} \|\vartheta\|_1 = \frac{1+a}{1-a} \|\beta^Y - \hat{\beta}^Y\|_1,
\]
giving the result.

We first observe that by Lemma~2 of \citet{Belloni2011} and also equation (13) therein, for any $B>1$,
\begin{equation} \label{eq:sqrt_lasso_bd}
	\pr\left(\frac{\|\check{\mb Z}^T\mbb\varepsilon^Y\|_\infty / n}{\|\mbb\varepsilon^Y\|_2/\sqrt{n}} \leq B\sqrt{\frac{2\log p}{n}}\right) \to 1.
\end{equation}
Now by Lemma~3.1 of \citet{van2016estimation},
writing
\[
\hat{\delta} := 2\sqrt{\left(\frac{\lambda \|\beta^Y - \hat{\beta}^Y\|_1}{2\|\mbb\varepsilon^Y\|_2/\sqrt{n}} + 1\right)^2 -1},
\]
we have that the event
\begin{equation*}
	\Omega_{1n} := \left\{\frac{\|\check{\mb Z}^T\mbb\varepsilon^Y\|_\infty / n}{\|\mbb\varepsilon^Y\|_2/\sqrt{n}} \leq (1-\hat{\delta})\lambda\right\}
\end{equation*}
satisfies $\Omega_{1n} \subseteq \Omega_{2n}$ given by
\[
\Omega_{2n} := \left\{\frac{\|\mbb\varepsilon^Y\|_2 / \sqrt{n}}{\check{\sigma}} \leq \frac{1}{1-\hat{\delta}}\right\}.
\]
Next for any $1\geq a > 1/A$, we have writing
\[
\Omega_{3n} := \left\{\frac{\|\check{\mb Z}^T\mbb\varepsilon^Y\|_\infty / n}{\|\mbb\varepsilon^Y\|_2/\sqrt{n}} \leq a(1-\hat{\delta})\lambda\right\},
\]
that $\Omega_{3n} \subseteq \Omega_{1n}$. Thus on $\Omega_{3n}$ we have that \eqref{eq:a_bd} holds.

Now by the weak law of large numbers and the continuous mapping theorem, $\|\mbb\varepsilon^Y\|_2 / \sqrt{n} \inprob \sigma$. Moreover $\lambda \|\beta^Y - \hat{\beta}^Y\|_1 \inprob 0$ due to (Y3) and (Y4). Thus from \eqref{eq:sqrt_lasso_bd} we see that $\pr(\Omega_{3n}) \to 1$, proving the first part of the result. 
The second part of the result is an easy consequence of the first and follows from the same arguments as used to prove Theorem~\ref{thm:DEF_lin}. \qed

\section{Computation of the square-root Lasso} \label{sec:square-root_Lasso}
Here we explain how the square-root Lasso
\begin{equation} \label{eq:sqrtlasso}
	\hat{\beta}^{\sq}_\lambda := \argmin_{\beta \in \R^p} \left\{\frac{1}{\sqrt{n}} \|\mb Y - \mb Z \beta\|_2 + \lambda \|\beta\|_1 \right\}
\end{equation}
may be computed easily given regular Lasso solutions
\begin{equation} \label{eq:lasso}
	\hat{\beta}^{\re}_\gamma := \argmin_{\beta \in \R^p} \left\{\frac{1}{2n} \|\mb Y - \mb Z \beta\|_2^2 + \gamma \|\beta\|_1 \right\}.
\end{equation}
As we will see, a square-root Lasso solution path may be derived from any Lasso solution path via a non-decreasing reparametrisation of the tuning parameter.

Now the minimisers $\hat{\beta}^{\sq}(\lambda)$ and $\hat{\beta}^{\re}(\gamma)$ need not be unique,
%but under reasonable conditions \citep[Lem.~4]{tibshirani2013lasso} the latter will be unique for all $\gamma>0$; we tacitly assume this in the following discussion. 
but the fitted values $\mb Z \hat{\beta}^{\re}(\gamma) $ of the regular Lasso are always unique.
%is unique provided $\lambda > \lambda^*$.
To see this, observe that fixing $\gamma \geq 0$ and taking $\beta^{(1)}$ and $\beta^{(2)}$ as two solutions to \eqref{eq:lasso} achieving minimum value $c^*$, we have due to the triangle inequality and strict convexity of $\| \cdot \|_2^2$ that
\begin{align}
	c^* &\leq \frac{1}{2n} \|\mb Y - \mb Z (\beta^{(1)} + \beta^{(2)})/2\|_2^2 + \lambda( \|(\beta^{(1)}+\beta^{(2)})/2\|_1 ) \notag\\
	&\leq \frac{1}{2} \frac{1}{2n} (\|\mb Y - \mb Z \beta^{(1)}\|_2^2 + \|\mb Y - \mb Z \beta^{(2)}\|_2^2)  \\
	& \;\;\;\; + \frac{\lambda}{2}( \|\beta^{(1)}\|_1 + \|\beta^{(2)}\|_1 ) = c^*. \notag %\label{eq:unique_arg}
\end{align}
Thus equality must hold throughout, which can only be the case if $\beta^{(1)} = \beta^{(2)}$.

Let us write
\begin{align*}
	\hat{\sigma}^{\sq}_\lambda &:= \frac{1}{\sqrt{n}} \| \mb Y -\mb Z\hat{\beta}^{\sq}_\lambda\|_2, \\
	\hat{\sigma}^{\re}_\gamma &:= \frac{1}{\sqrt{n}} \| \mb Y -\mb Z\hat{\beta}^{\re}_\gamma\|_2;
\end{align*}
note that as the fitted values are unique, the latter is uniquely defined though the former may not be.

To establish the relationship between the Lasso and square-root Lasso solutions, observe that the KKT conditions of \eqref{eq:sqrtlasso} and \eqref{eq:lasso} are given by
\begin{align*}
	\frac{1}{n\hat{\sigma}^{\sq}_\lambda} \mb Z^T(\mb Y-\mb Z\hat{\beta}_\lambda^{\sq}) &= \lambda \hat{\nu}^\sq_\lambda,\\
	\frac{1}{n} \mb Z^T(\mb Y-\mb Z\hat{\beta}^{\re}_\gamma) &= \gamma \hat{\nu}^\re_\gamma,
\end{align*}
where $\|\hat{\nu}^\sq_\lambda\|_\infty \leq 1$, and $\hat{\nu}^\sq_\lambda$ agrees in sign with $\hat{\beta}_\lambda^{\sq}$ on its active set (and similarly for $\hat{\nu}^\re_\gamma$), provided $\hat{\sigma}^{\sq}_\lambda > 0$.
Comparing the KKT conditions above, we see that any Lasso solution $\hat{\beta}^\re_\gamma$ is a square-root Lasso solution $\hat{\beta}^\sq_\lambda$ with $\lambda = \gamma / \hat{\sigma}^\re_\gamma$ (provided $\hat{\sigma}^\re_\gamma > 0$). Conversely, any square-root Lasso solution $\hat{\beta}^\sq_\lambda$ is equal to a Lasso solution $\hat{\beta}^\re_\gamma$ with $\gamma = \lambda \hat{\sigma}^\sq_\lambda$, provided $\hat{\sigma}^{\sq}_\lambda > 0$.

\begin{lem} \label{lem:sqrt_lasso}
	Let $\gamma^*$ be maximal such that $\hat{\sigma}^{\re}_{\gamma^*}=0$.
	The function $\gamma \mapsto \gamma / \hat{\sigma}^{\re}_\gamma$ defined on $(\gamma^*, \infty)$ is non-decreasing.
\end{lem}

The result above shows that given a square-root Lasso tuning parameter $\lambda$, we may find via a bisection search the Lasso tuning parameter $\gamma$ such that $ \gamma / \hat{\sigma}^{\re}_\gamma = \lambda$ and thereby obtain a square-root Lasso solution.

\subsection{Proof of~Lemma~\ref{lem:sqrt_lasso}}
The conclusion is equivalent to the following: for any $\gamma_1, \gamma_2 \in (\gamma^*, \infty)$ with
\[
\lambda_1:=\frac{\gamma_1}{\hat{\sigma}^{\re}_{\gamma_1}} < \frac{\gamma_2}{\hat{\sigma}^{\re}_{\gamma_2}}=:\lambda_2,
\]
we have $\gamma_1 < \gamma_2$.
Let us write $\beta^{(1)} = \hat{\beta}^{\re}_{\gamma_1}$ and $\beta^{(2)} = \hat{\beta}^{\re}_{\gamma_2}$, noting that whilst these need not be unique, the corresponding fitted values and $\ell_1$-norms are. Then as $\beta^{(1)}$ and $\beta^{(2)}$ are square-root Lasso solutions at $\lambda_1$ and $\lambda_2$ respectively, we have that
\begin{align}
	& \frac{1}{\sqrt{n}} \|\mb Y - \mb Z \beta^{(1)} \|_2 + \lambda_1 \|\beta^{(1)}\|_1 \notag \\
	\leq & \frac{1}{\sqrt{n}} \|\mb Y - \mb Z \beta^{(2)} \|_2 + \lambda_1 \|\beta^{(2)}\|_1 \label{eq:eq1} \\
	&\frac{1}{\sqrt{n}} \|\mb Y - \mb Z \beta^{(2)} \|_2 + \lambda_2\|\beta^{(2)}\|_1 \notag \\
	\leq & \frac{1}{\sqrt{n}} \|\mb Y - \mb Z \beta^{(1)}\|_2 + \lambda_2 \|\beta^{(1)}\|_1. \notag
\end{align}
Adding these inequalities, we deduce that
\[
\lambda_1 \|\beta^{(1)}\|_1 + \lambda_2 \|\beta^{(2)}\|_1 \leq  \lambda_1 \|\beta^{(2)}\|_1 + \lambda_2 \|\beta^{(1)}\|_1.
\]
Rearranging, we obtain
\[
(\lambda_2 - \lambda_1)(\|\beta^{(2)}\|_1 - \|\beta^{(1)}\|_1) \geq 0,
\]
and so dividing by $\lambda_2 - \lambda_1 >0$ we conclude that $\|\beta^{(2)}\|_1 \geq \|\beta^{(1)}\|_1$. Substituting this into \eqref{eq:eq1}, we see that $\hat{\sigma}^{\re}_{\gamma_1} \leq \hat{\sigma}^{\re}_{\gamma_2}$, so $\gamma_1 = \hat{\sigma}^{\re}_{\gamma_1} \lambda_1 < \hat{\sigma}^{\re}_{\gamma_2} \lambda_2 = \gamma_2$ as required. \qed

\section{Confidence regions for $w^T\beta^0$} \label{sec:pred_int}
In this section we consider a linear model $\mb Y = \mb Z \beta^0 + \mbb\varepsilon$ and consider the problem of finding a confidence interval for $w^T\beta^0$ for a given $w \in \R^p$. When $w = e_j$ for a standard basis vector $e_j \in \R^p$, the methodology set out in Section~\ref{sec:high_conf} may be used to obtain a confidence region even in the case where only a partially linear model holds. For more general $w$, these methods must be adapted and here we will need to assume the linear model above holds with $\beta^0$ sufficiently sparse. We describe these modifications below.
%\Rajen{Compare to section 5.3 in Sara's book?}

First consider testing a null hypothesis $H_0$: $w^T\beta^0=0$. Let $P=ww^T / \|w\|_2^2$. Note $w^T\beta^0=0$ if and only if $(I-P)\beta^0 = \beta^0$, so the null model may be expressed as
\begin{equation} \label{eq:pred_mod1}
	\mb Y = (I-P)\mb Z\beta^0 + \mbb\varepsilon.
\end{equation}
Let
\begin{equation} \label{eq:Y_pred_reg}
	\hat{\beta} = \argmin_{\beta \in \R^p} \left\{  \|\mb Y - \mb Z (I-P)\beta\|_2 / \sqrt{n} + \lambda \|\beta\|_1\right\}.
\end{equation}
Note that under $H_0$ we should have
\[
\|\hat{\beta} - \beta^0\|_1 = O_{\pr}(s \sqrt{\log(p)/n})
\]
for $\lambda =A \sqrt{2\log(p)/n}$ with $A>1$ and where $s = |\{j : \beta^0_j \neq 0\}|$.
Also let $\mb R \in \R^n$ be the vector of residuals from the regression
\begin{equation} \label{eq:w_est}
	\argmin_{\beta\in \R^p} \left\{ \|\mb Z w - \mb Z (I-P) \beta\|_2 / \sqrt{n} + \lambda \|\beta\|_1\right\}.
\end{equation}
Note that $\mb R$ thus defined enjoys a near-orthogonality property of the form $(I-P)\mb Z^T \mb R / \|\mb R\|_2 \leq \sqrt{n} \lambda$. The reason for aiming to orthogonalise $\mb Z w$ is that were we to have $w^T\beta^0 \neq 0$, the residuals from the regression \eqref{eq:Y_pred_reg} should have expectation close to $\mb Z P \beta^0 \propto \mb Z w$. Thus a test statistic involving dotting these residuals with something close to the direction of $\mb Z w$ should be large in magnitude under an alternative.

We thus consider the test statistic given by
\begin{equation} \label{eq:predT}
	T = \sqrt{n}\frac{\mb R^T\{\mb Y - \mb Z (I-P)\hat{\beta}\}}{\|\mb R\|_2\|\mb Y - \mb Z (I-P)\hat{\beta}\|_2} .
\end{equation}
Writing $\hat{\sigma} = \|\mb Y - \mb Z(I-P)\hat{\beta}\|_2 / \sqrt{n}$, we have
\begin{align*}
	T &= \frac{1}{\hat{\sigma}}\frac{\mb R^T}{\|\mb R\|_2} \mbb\varepsilon  + \frac{1}{\hat{\sigma}} (\hat{\beta}-\beta^0)^T(I-P)\mb Z^T \frac{\mb R}{\|\mb R\|_2}\\
	&=: \text{(i)}+\text{(ii)}.
\end{align*}
Term (i) will be well-approximated by a standard normal under reasonable conditions, and term (ii) may be bounded in absolute value using an argument similar to that presented in Section~\ref{sec:high_lin}.
Thus under appropriate conditions, we will have $T \indist \mathcal{N}(0, 1)$.

Now consider testing $H_0(t)$: $w^T\beta^0 = t$. Observe that
\[
\mb Y - t \mb Z w / \|w\|_2^2 = \mb Z \beta^0 - \mb Z P \beta^0 + \mbb\varepsilon =: \mb Y^{(t)},
\]
so the new response $\mb Y^{(t)}$ respects the null model \eqref{eq:pred_mod1}. We may thus test $H_0(t)$ using test statistic $T_t$ defined as in \eqref{eq:predT} but computed using the response $\mb Y^{(t)}$ in place of $\mb Y$.
%Now set 
%\[
%\hat{\beta} = \argmin_{\beta \in \R^p} \frac{1}{\sqrt{n}} \|\mb Y - t \mb Zw  / \|w\|_2^2 - \mb Z (I-P)\beta\|_2 + \lambda \|\beta\|_1.
%\]
%Again, it may be shown that $\|\hat{\beta} - \beta^0\|_1 = O_P(s \sqrt{\log(p)/n})$, and under appropriate conditions our test statistic will be asymptotically normal.

Then to form a $1-\alpha$ confidence region for $w^T\beta^0$ we can simply invert the tests as in Section~\ref{sec:high_conf}:
\[
R_\alpha := \{t \in \R : |T_{t}| \geq z_{\alpha}\}.
\]
Provided $\pr( H_0(w^T\beta^0) \text{ rejected}) \geq 1-\alpha$, the confidence region $R_\alpha$ will satisfy $\pr(w^T\beta^0 \in R_\alpha) \geq 1-\alpha$; see Corollary~\ref{cor:conf}.

We note that compared to the confidence regions constructed in \citet{cai2017confidence}, which are introduced primarily for theoretical purposes, our confidence region does not require prior knowledge of the the inverse covariance of $\mb Z$, the sparsity of $\beta^0$, or the noise level $\Var(\varepsilon_1)$. Our construction is related to that in \citet{zhu2018linear}, but where we require sparsity of $\beta^0$, \citet{zhu2018linear} instead require sparsity of a projection of the quantity `estimated' by the minimiser in \eqref{eq:w_est}. In fact, with such an assumption, it is straightforward to see that we can still expect $T$ to have an asymptotically normal distribution regardless of the sparsity of $\beta^0$ by reversing the roles of \eqref{eq:Y_pred_reg} and \eqref{eq:w_est}: we only use the former to establish approximate orthogonality while we exploit assumed small estimation error of the latter.
%However, we view such a sparsity assumption as somewhat unnatural compared to the sparsity of $\beta^0$.
An additional difference is that the approach in \citet{zhu2018linear} requires solving a family of large-scale linear programs, whereas our region requires only standard software for computing the Lasso.

\section{Additional numerical results} \label{sec:add_exp}
Here we present the results of analogous numerical experiments to those described in Section~\ref{sec:exp_lin}, but with the multivariate distribution $P$ used for generating predictors $(X_i, Z_i)$ replaced with a multivariate Gaussian distribution $\mathcal{N}_p(0, \Sigma)$. We take the covariance matrix $\Sigma$ to have a Toeplitz design with $\Sigma_{jk} = 0.9^{|j-k|}$.  Note that the inverse of $\Sigma$ is tridiagonal and so the $X$-model is a sparse linear model (with sparsity level $s_X=1$). The settings considered here thus satisfy the conditions of Theorem~\ref{thm:DEF_lin}.

We see that compared to the more challenging settings investigated in Section~\ref{sec:exp_lin}, the coverage properties of both confidence interval construction methods are improved; however the debiased Lasso still undercovers whilst the DEF confidence intervals reach a coverage of closer to 95\%.
\begin{figure*}
	\centering
	\includegraphics[width=\textwidth]{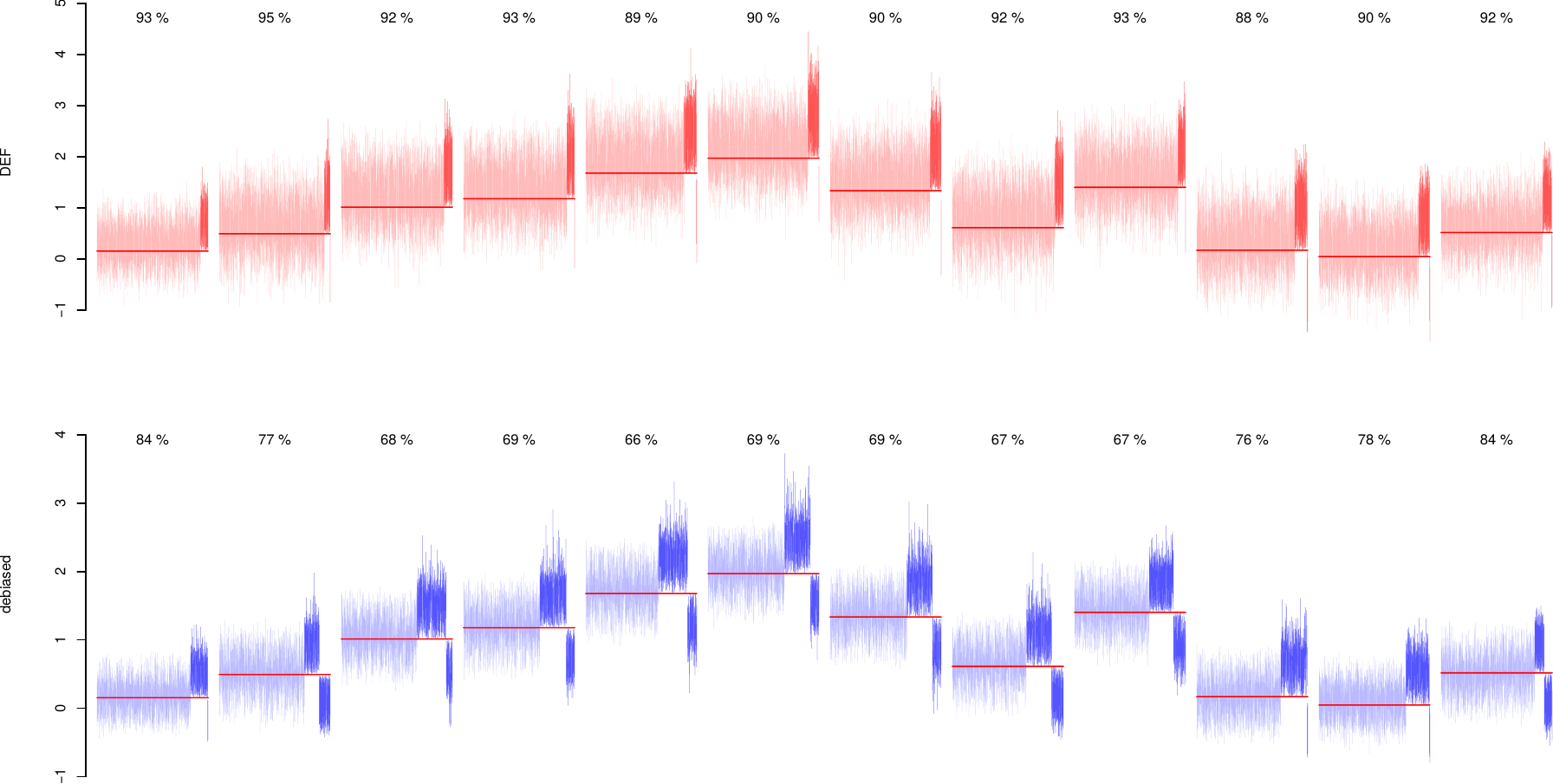}
	\caption{The linear setting (a) with Toeplitz design; the interpretation is similar to that of Figure~\ref{fig:n1_real_lin}.\label{fig:n1_toep_nonlin1}}
\end{figure*}
\begin{figure*} 
	\centering
	\includegraphics[width=\textwidth]{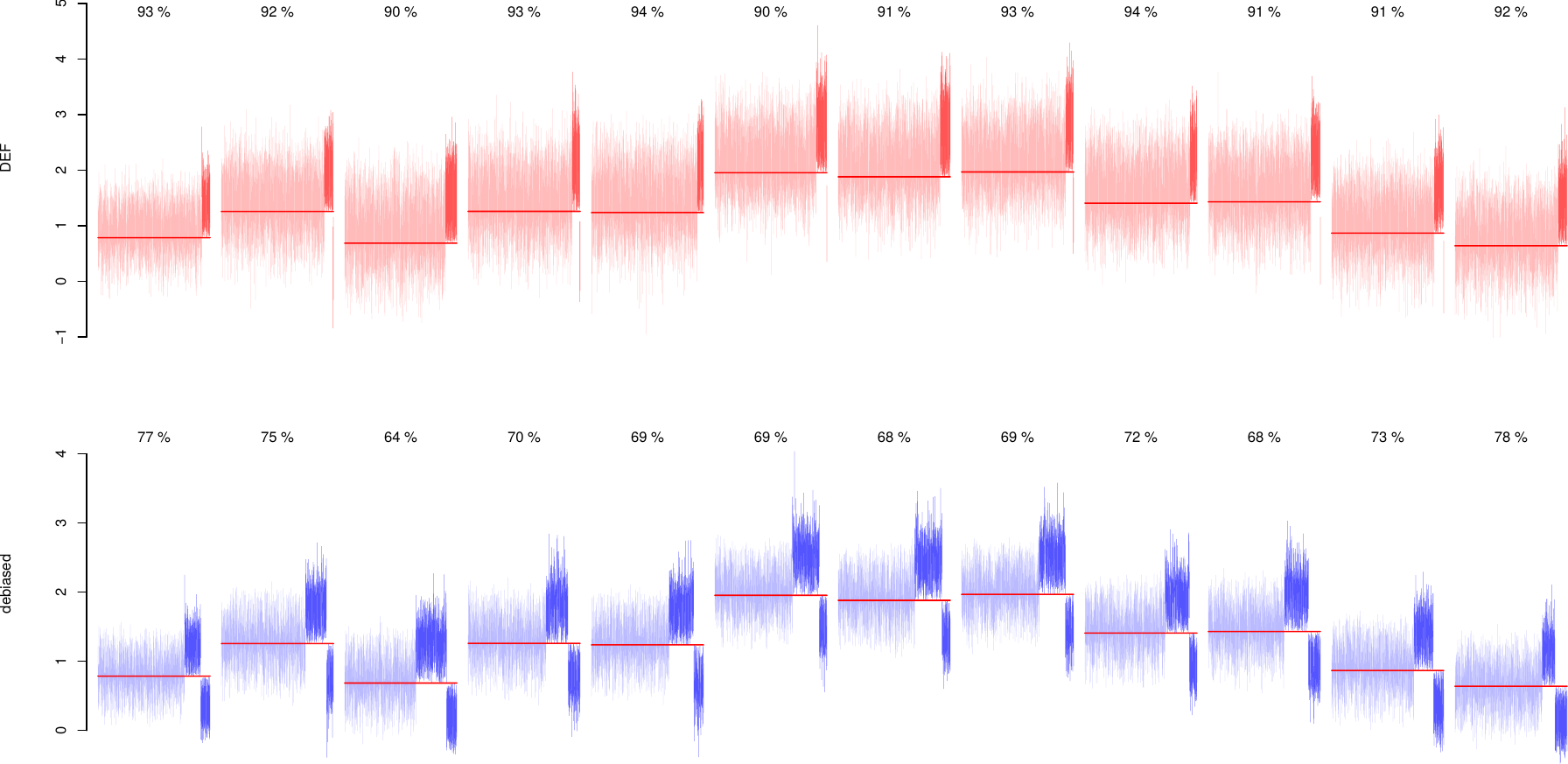}
	\caption{The slightly nonlinear setting (b) with Toeplitz design; the interpretation is similar to that of Figure~ \ref{fig:n1_real_lin}.\label{fig:n1_toep[_nonlin1}}
\end{figure*}
\begin{figure*} 
	\centering
	\includegraphics[width=\textwidth]{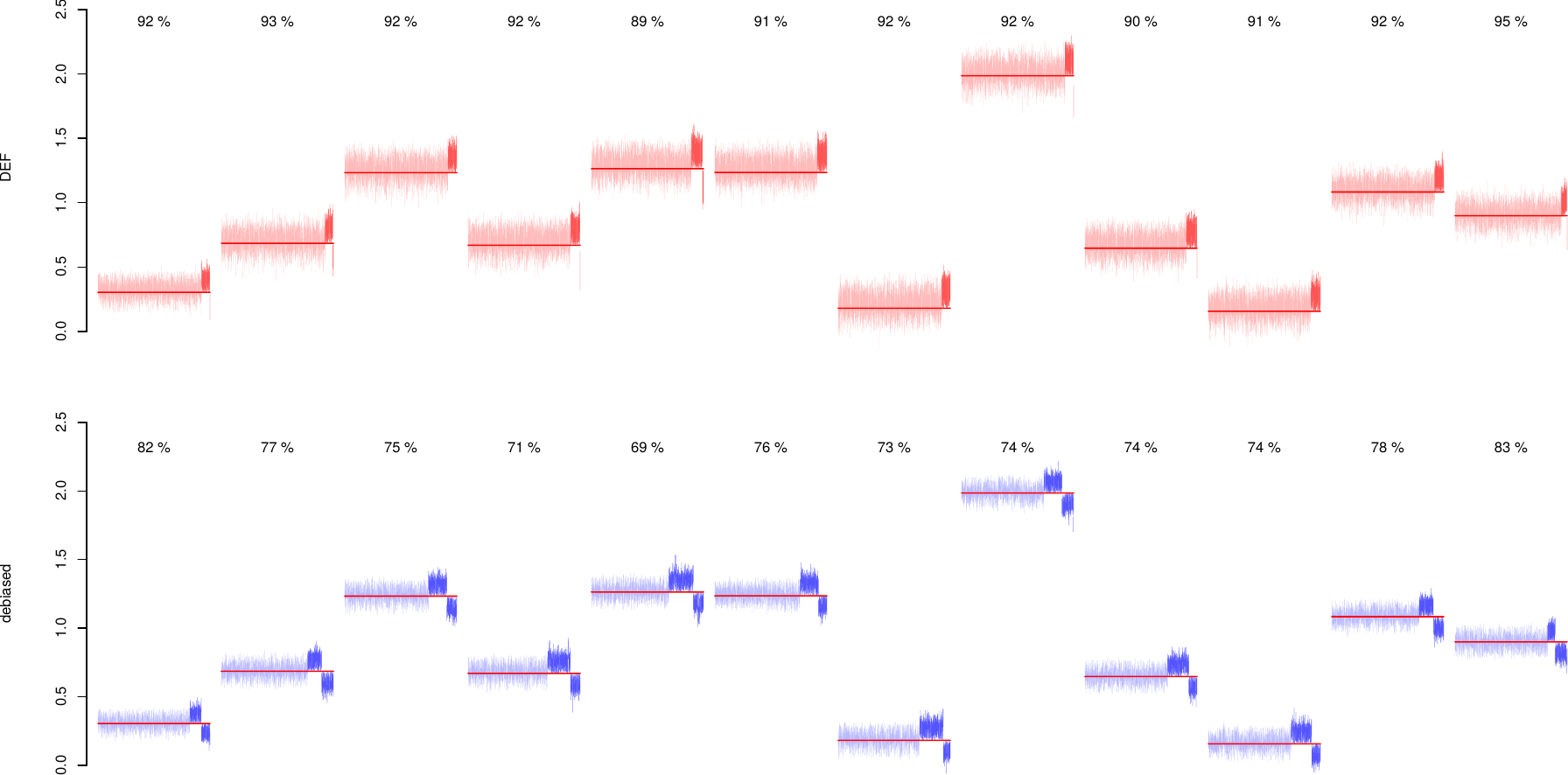}
	\caption{The highly nonlinear setting (c) with Toeplitz design; the interpretation is similar to that of Figure~\ref{fig:n1_real_lin}.\label{fig:n1_toep_nonlin2}}
\end{figure*}
\end{appendix}

%%%%%%%%%%%%%%%%%%%%%%%%%%%%%%%%%%%%%%%%%%%%%%
%% Support information, if any,             %%
%% should be provided in the                %%
%% Acknowledgements section.                %%
%%%%%%%%%%%%%%%%%%%%%%%%%%%%%%%%%%%%%%%%%%%%%%
\begin{acks}[Acknowledgments]
The authors would like to thank Nicolai Meinshausen for many helpful conversations, and also coining the term ``double estimation friendly''.
\end{acks}
%%%%%%%%%%%%%%%%%%%%%%%%%%%%%%%%%%%%%%%%%%%%%%
%% Funding information, if any,             %%
%% should be provided in the                %%
%% funding section.                         %%
%%%%%%%%%%%%%%%%%%%%%%%%%%%%%%%%%%%%%%%%%%%%%%
\begin{funding}
 The first author was supported by an EPSRC Programme Grant EP/N031938/1 and an EPSRC First Grant EP/R013381/1.

 The second author was supported by the European Research Council
 under the Grant Agreement No 786461 (CausalStats - ERC-2017-ADG).
\end{funding}

%%%%%%%%%%%%%%%%%%%%%%%%%%%%%%%%%%%%%%%%%%%%%%
%% Supplementary Material, including data   %%
%% sets and code, should be provided in     %%
%% {supplement} environment with title      %%
%% and short description. It cannot be      %%
%% available exclusively as external link.  %%
%% All Supplementary Material must be       %%
%% available to the reader on Project       %%
%% Euclid with the published article.       %%
%%%%%%%%%%%%%%%%%%%%%%%%%%%%%%%%%%%%%%%%%%%%%%
%\begin{supplement}
%\stitle{???}
%\sdescription{???.}
%\end{supplement}

%%  MR numbers will be added by VTeX.                      %%
%%                                                         %%
%%  Use \cite{...} to cite references in text.             %%
%%                                                         %%
%%%%%%%%%%%%%%%%%%%%%%%%%%%%%%%%%%%%%%%%%%%%%%%%%%%%%%%%%%%%%

%% if your bibliography is in bibtex format, uncomment commands:
\bibliographystyle{imsart-nameyear.bst} % Style BST file (imsart-number.bst or imsart-nameyear.bst)
\bibliography{references}       % Bibliography file (usually '*.bib')

%% or include bibliography directly:
% \begin{thebibliography}{}
% \bibitem{b1}
% \end{thebibliography}

\end{document}